\documentclass[cmp,referee]{svjour}

\usepackage{amsfonts}
\usepackage{amssymb,amsmath, enumerate}
\usepackage{latexsym}
\usepackage{fullpage}
\usepackage{hyperref}

\newtheorem{prop}[theorem]{Proposition}
\newtheorem{cor}[theorem]{Corollary}

\newcounter{tenumerate}

\def\P{\mathbb{P}}

\newcommand{\deq}{\stackrel{\scriptscriptstyle\triangle}{=}}

\renewcommand{\epsilon}{\varepsilon}

\newcommand{\1}{\mathbf{1}}

\DeclareMathOperator{\var}{Var}

\newcommand{\E}{{\mathbb E}}
\newcommand{\remove}[1]{}

\renewcommand{\leq}{\leqslant}
\renewcommand{\geq}{\geqslant}

\newcommand{\temp}{{\mathrm{Explored}}_k}

\def\XXint#1#2#3{{\setbox0=\hbox{$#1{#2#3}{\int}$}
		\vcenter{\hbox{$#2#3$}}\kern-.5\wd0}}

\allowdisplaybreaks

\begin{document}
\title{Chemical distances for percolation of planar Gaussian free fields and critical random walk loop soups}
\author{Jian Ding$^{\small{1}}$ \and Li Li$^{\small{2}}$}
\institute{Department of Statistics, University of Pennsylvania, Philadelphia, PA 19104, USA; Email: dingjian@wharton.upenn.edu; Phone: (215) 898-8222. \and Department of Statistics, University of Chicago, Chicago, IL 60637, USA.}
\date{Received: date / Accepted: date}

\communicated{name}

\maketitle
\begin{abstract}
We initiate the study on chemical distances of percolation clusters for level sets of two-dimensional discrete Gaussian free fields as well as loop clusters generated by two-dimensional random walk loop soups. One of our results states that the chemical distance  between two macroscopic annuli away from the boundary for the random walk loop soup at the critical intensity is of dimension 1 with positive probability. Our proof method is based on an interesting combination of a theorem of Makarov, isomorphism theory and an entropic repulsion estimate for Gaussian free fields in the presence of a hard wall.
\end{abstract}

\section{Introduction}
For $N\geq 1$, let $V_N\subseteq \mathbb Z^2$ be an $N\times N$ box  centered at the origin or at $(1/2, 1/2)$ depending on the parity of $N$. We define the discrete Gaussian free field (GFF) $\{\eta_{N,v}: v\in V_N\}$ with Dirichlet boundary condition to be a mean zero Gaussian
process which takes value 0 on $\partial V_N = \{v\in  V_N: u\sim v \mbox{ for some } u\in  \mathbb Z^2\setminus V_N\}$ and has covariances given by
$$\E \eta_{N, v} \eta_{N, u} =  \frac{1}{4}G_{V_N}(u, v) \mbox{ for } u, v\in V_N\,,$$
where $G_{V_N}(u, v)$ is the Green's function for simple random walk, i.e., the expected number of visits to $v$ before reaching $\partial V_N$ for a simple random walk started at $u$.
The first goal of the present paper is to study chemical distances (i.e., graph distances) on percolation clusters for level sets of GFFs. Precisely, for any $\lambda\in \mathbb R$, we let $\mathcal H_{N, \lambda} = \{v\in V_N: \eta_{N,v} \leq \lambda\}$ be the $\lambda$-level set, i.e., the collection of all vertices with values no more than  $\lambda$. In the context of no confusion, we also denote by $\mathcal H_{N, \lambda}$ the induced subgraph on $\mathcal H_{N, \lambda}$. For $u, v\in V_N$, we let $D_{N, \lambda}(u, v)$ be the graph distance in $\mathcal H_{N, \lambda}$ between $u$ and $v$ if $u, v$ are in the same connected component of $\mathcal H_{N, \lambda}$, and let $D_{N, \lambda}(u, v) = \infty$ otherwise. For $A, B\subseteq V_N$, we denote $D_{N, \lambda}(A, B) = \min_{u\in A, v\in B} D_{N, \lambda}(u, v)$. In addition, we abuse the notation by denoting $V_x = V_{\lfloor x \rfloor }$ for $x\geq 1$, where $\lfloor x \rfloor$ is the greatest integer that is at most $x$. 
\begin{theorem}\label{main-thm-annulus}
	For any $0<\alpha<\beta < 1$ and $\chi>1/2$, there exist constants $c>0, \lambda_0>0$ such that
	$$\P(D_{N, \lambda}(\partial V_{\alpha N}, \partial V_{\beta N}) \geq N \mathrm{e}^{(\log N)^{\chi}}) \leq c^{-1} (\mathrm{e}^{-c\lambda^2} + N^{-20}), \mbox{ for all } \lambda \geq \lambda_0 \mbox{ and } N\geq 1\,.$$
\end{theorem}
\begin{remark}\label{rem-main}
	Note that even for any fixed $\lambda<0$, the event $D_{N, \lambda}(\partial V_{\alpha N}, \partial V_{\beta N}) \leq N \mathrm{e}^{(\log N)^{\chi}}$ in Theorem~\ref{main-thm-annulus} occurs with non-vanishing probability; see Corollary~\ref{cor-main-thm}. In addition, we expect that for any fixed $\lambda$, the probability for $D_{N, \lambda}(\partial V_{\alpha N}, \partial V_{\beta N}) <\infty$ is strictly less than 1; we do not study this in the present  paper so as not to dilute the focus.  \end{remark}

We next consider the random walk loop soup introduced in \cite{LT07}, which is a discrete analogue of the Brownian loop soup \cite{LW04}. For convenience, we follow \cite{LeJan11} where the loops are endowed with a continuous-time parametrization. Formally, let $(X_t)$ be a continuous-time sub-Markovian jump process on $V_N$ which is killed at the boundary $\partial V_N$. Given two neighboring vertices $x$ and $y$, let the transition rate from $x$ to $y$ be 1. Let $(\mathbb P^t_{x, y}(\cdot))_{x, y\in V_N, t>0}$ be the bridge probability measures of $X$ conditioned on not killed until time $t$, and let $(p_t(x, y))_{x, y\in V_N, t\geq 0}$ be the transition probabilities of $X$. Then the measure $\mu$ on time-parametrized loops associated to $X$ is, as defined in \cite{LeJan11},
\begin{equation}\label{eq-def-loop-measure}
	\mu(\cdot) = \sum_{x\in V_N} \int_0^\infty \mathbb P_{x, x}^t(\cdot) \frac{p_{t}(x, x)}{t} dt\,.
\end{equation}
For $\alpha>0$, the random walk loop soup with intensity $\alpha$ on $V_N$, denoted as $\mathcal L_{\alpha, N}$, is defined to be the Poisson point process on the space of loops with intensity $\alpha \mu$.   Naturally $\mathcal L_{\alpha, N}$ induces a subgraph (which we also denote as $\mathcal L_{\alpha, N}$) of $\mathcal G_N$ where an edge is open if it is contained in (at least) one loop in $\mathcal L_{\alpha, N}$. We are particularly interested in the critical random walk loop soup, i.e., when $\alpha = \alpha_c = 1/2$. 
We denote by $D_{\mathcal L_{1/2, N}}(\cdot, \cdot)$ the chemical distance (i.e., graph distance) on the induced subgraph $\mathcal L_{1/2, N}$ (as above, we use the convention that $D_{\mathcal L_{1/2, N}}(u, v) = \infty$ if $u$ and $v$ are disconnected in $\mathcal L_{1/2, N}$).
\begin{theorem}\label{thm-random-walk-loop-soup-critical}
	For any $0<\alpha<\beta < 1$ and $\chi>1/2$, there exists a constant $c>0$ such that for all $N$
	$$\P(D_{\mathcal L_{1/2, N}}(\partial V_{\alpha N}, \partial V_{\beta N}) \leq N \mathrm{e}^{(\log N)^{\chi}}) \geq c\,.$$
\end{theorem}
\begin{remark}
	We expect that the probability for $D_{ {\mathcal L}_{1/2, N}}(\partial V_{\alpha N}, \partial V_{\beta N}) < \infty$ is strictly less than 1; see Remark~\ref{rem-main}. 
\end{remark}

\subsection{Backgrounds and related works}

Chemical distances for percolation models is a substantially more challenging problem than the question on connectivities. For instance, it is a major challenge to compute the exponent on the chemical distance between (say) the left and right boundaries for the critical planar percolation, conditioned on the existence of an open crossing. It was proved in \cite{AB99} that the dimension is strictly larger than 1, and it was shown in recent works \cite{DHS15,DHS17} that the chemical distance is substantially smaller than the length of the lowest open crossing --- indeed it was shown in \cite{DHS17} that the exponent for the chemical distance is strictly less than that of the lowest open crossing.

Due to the strong correlation and hierarchical nature of the two-dimensional GFF as well as the random walk loop soup, our models are perhaps in spirit more closely related to the fractal percolation process (see \cite{Chayes95} for a survey). For fractal percolation process, it was proved \cite{Chayes96,Orz98} that the dimension of the chemical distance is strictly larger than 1 (which suggests an interesting dichotomy in view of our dimension 1 results for the GFF and the random walk loop soup).

As for loop soups, in two-dimensions the connectivity of the loop clusters has been studied recently. In \cite{SW12}, it was shown that there is a phase transition around the critical intensity $\alpha_c = \frac{1}{2}$ for percolation of the Brownian loop soup, below which there are only bounded clusters and above which the loops forms a single cluster. In recent works of \cite{Lupu14,Lupu14b}, analogous results were proved for the random walk loop soup.

In three-dimensions or higher, there has been an intensive study on percolation of level sets for GFFs, random walks,  random interlacements as well as random walk loop soups; see, e.g., \cite{SS09,Sznitman10,RS13,CS16}. In fact, much on the chemical distances for these percolation models has been studied; see \cite{CP12,DRS14,Chang15}.
We remark that there is a drastic difference between two-dimensions and higher dimensions. 

Besides chemical distances, other metric aspects of two-dimensional GFF has been studied recently: see \cite{LW16} on the random pseudo-metric defined via the zero-set, and see   \cite{DZ15,DG15,DD16,DG16} for some progress on the first passage percolation on the exponential of these underlying fields.

Finally, the random walk loop soup percolation is naturally related to the following percolation dimension question for planar random walks (Brownian motion) proposed in \cite{DLLL93,Burdzy}. Run the random walk until it exits the boundary of a box and declare a vertex to be open if it is visited and closed otherwise.  Then what is the dimension of the minimal open crossing from the origin to the boundary? We are currently not able to prove anything for this question, for the crucial reason that we are not able to construct a coupling between GFFs and random walks under which events on GFFs will certify ``small'' chemical distances for random walk percolation models.

\subsection{Discussions on main proof ingredients}\label{sec:proof-ingredients}

Our proofs of Theorems~\ref{main-thm-annulus} and \ref{thm-random-walk-loop-soup-critical}  are based on an interesting combination of a theorem of Makarov, isomorphism theory and an entropic repulsion estimate for GFF in the presence of hard wall. In this subsection, we will provide a brief review on these three ingredients.

\medskip

\noindent{\bf A theorem of Makarov.} A fundamental ingredient for our proofs, is a classical theorem of Makarov \cite{Makarov85}  which states that the dimension of the support for the harmonic measure on simply connected domain in $\mathbb R^2$ is 1. In this article, we will use the following discrete analogue of Makarov's theorem which was proved in \cite{Lawler93} by approximating Brownian motions with random walks (and then using \cite{Makarov85}). For $u, v\in \mathbb Z^2$ and $A\subset \mathbb Z^2$, we use $\mathrm{Hm}(v, u; A)$ to denote the harmonic measure at $u$ with respect to starting point $v$ and the target set $A$ (i.e., $\mathrm{Hm}(v, u; A) = \P_v( S_{\tau_A} = u)$, where $(S_n)$ is a simple random walk on $\mathbb Z^2$ and $\tau_A$ is the first time it hits set $A$). In addition, we denote $\mathrm{Hm}(\infty, u; A) = \lim_{|v|_1\to \infty} \mathrm{Hm}(v, u; A)$ (the existence of the limit is well-known; c.f. \cite{Spitzer76}), where we denoted by $|\cdot|_1$ the $\ell_1$-norm. We further denote by $\mathrm{Hm}(\infty, B; A) = \sum_{u\in B} \mathrm{Hm}(\infty, u; A)$. 
\begin{theorem}\cite[Proposition 4.1]{Lawler93}\label{thm-Makarov}
For any $\chi>1/2$ and $\chi'<\infty$, there exists a positive constant $C$ depending only on $(\chi, \chi')$ such that for all $n\geq 1$ and any connected subset $A\subseteq \mathbb Z^2$ of diameter  $n$ (measured in $\ell_1$-distance) we have
$$\mathrm{Hm}(\infty, \{x\in A: \mathrm{Hm}(\infty, x; A) \geq n^{-1} \mathrm{e}^{(\log n)^{\chi}}\}; A) \leq C (\log n)^{-\chi'}\,.$$
\end{theorem}
\begin{remark}
The assumption of $\chi>1/2$ in Theorem~\ref{thm-Makarov} is responsible for the same condition on $\chi$ in Theorems~\ref{main-thm-annulus} and \ref{thm-random-walk-loop-soup-critical}.
\end{remark}
Previous to \cite{Makarov85}, the Beurling's projection theorem (see, e.g., \cite[Theorem V.4.1]{Bass95}, and see \cite{Kesten87,LL04} for its discrete analogue) was established, which gives an (achievable) upper bound on the maximal local expansion of the harmonic measure compared with 1-dimensional Hausdorff measure (in the language of simple random walk, it states that the harmonic measure at a lattice point on a simply connected set of diameter $n$ is bounded by $O(1/\sqrt{n})$). In a sense, Makarov's theorem states that the upper bound in Berling's estimate cannot be achieved globally, and thus providing a much better control (than that guaranteed by Beurling's projection theorem) on the global expansion and compression of harmonic measure. Finally, we remark that examples have been given in \cite{MP73,Carleson73}, in which the harmonic measure is \emph{singular} to the 1-dimensional Hausdorff measure. In our opinion, this suggests that Question~\ref{question-super-linear} below could be of serious challenge.

\medskip

\noindent{\bf Isomorphism theory.} The distribution of the occupation times for random walks can be fully characterized by Gaussian free fields;
results of this flavor go by the name of isomorphism theorems (see \cite{MR06,LeJan11,Sznitman12,Rosen14} for an excellent account on this topic). Of significance to the present article is the following version of isomorphism theorem between occupation times for random walk loop soups and Gaussian free fields shown in \cite{LeJan11}.

Recall the definition of random walk loop soups $\mathcal L_{\alpha, N}$. We define the associated occupation time field $(\hat {\mathcal L}^x_\alpha)_{x\in V_N}$ by 
$$\hat {\mathcal L}^x_\alpha = \sum_{\gamma \in \mathcal L_{\alpha, N}} \int_0^{T(\gamma)} \mathbf 1_{\gamma(t)=x} dt$$
where $T(\gamma)$ is the duration of the loop $\gamma$. The isomorphism theorem in \cite{LeJan11} states that 
\begin{equation}\label{eq-isomorphism}
	\{{\hat {\mathcal L}^x_{1/2}}: x\in V_N\} \stackrel{law}{=} \{\frac{1}{2} \eta_{N, x}^2: x\in V_N\}
\end{equation}
(note that this holds for loop soups on general graphs). Couplings between random walks/random walk loop soups and Gaussian free fields have been developed recently in \cite{Lupu14}, where the signs of GFFs are incorporated in the coupling in order to provide certificate for vertices/edges not visited by random walks/random walk loop soups. The paper \cite{Lupu14} was motivated by connectivity of the loop soup clusters as well as random interlacement. Independent of \cite{Lupu14}, such coupling was established for random walks in \cite{Zhai14} with the application of deriving an exponential concentration for cover times. The work \cite{Zhai14} was motivated by \cite{Ding14}, where such coupling was proved for general trees and questioned for general graphs; the advance in \cite{Lupu14} was independent of \cite{Ding14}.

In fact, using the coupling derived in \cite{Lupu14} only allows us to prove a version of Theorem~\ref{thm-random-walk-loop-soup-critical} for Lupu's loop soup on the metric graph introduced in \cite{Lupu14}; see Section~\ref{sec-continuous-loop} and in particular Theorem \ref{thm-loop-soup-critical}. In order to deal with the random walk loop soup, we will use a more recent result on the \emph{random current model} for random walk loop soups. 
A random current model on a graph, say ${\mathcal G}_N=(V_N, E_N)$ in our case, is the probability measure $\P$ with 
\begin{equation}\label{eq-def-random-current}
	\P((n_e)_{e\in E_N})\propto \prod_{e\in E_N}\frac {(\beta_e)^{n_e}}{n_e!}\,,
\end{equation}
where $(n_e)_{e\in E_N}$ are nonnegative integers such that $\sum_{e:v\in e} n_e \text { is even for any } v\in V_N$, and $(\beta_e)_{e\in E_N}$ are positive parameters on $E_N$.  Conditioned on $\{\hat{\mathcal L}^v_{1/2} = \ell_v\}_{v\in V_N}$, let $(n_e)_{e\in E_N}$ be a random current model with parameters $\beta_e=2\sqrt{\ell_x\ell_y}$ on edge $e=(x,y)$. 

It was shown in \cite{Werner15,LW15,Lawler16}  (see \cite[Theorem 4 and Proposition 6.7]{Lawler16} for a formal statement) that conditioned on the local times the distribution of $(n_e)_{e\in E_N}$  is the same as that of the number of jumps of the random walk loop soup $\mathcal L_{1/2, N}$ along each $e\in E_N$, and therefore $(1_{n_e>0})_{e\in E_N}$ has the same distribution as the graph induced by $\mathcal L_{1/2, N}$ on $V_N$.

We remark that the random current representation played a crucial role in a recent work \cite{ADS15} which proved the continuity of spontaneous magnetization for the three-dimensional Ising model at the critical temperature. Finally, we remark that the random Eulerian graph model considered in \cite{Ding14} (which was used to reconstruct the number of visits to vertices from the continuous occupation times) was of high resemblance of  the random current model.

\medskip

\noindent {\bf Entropic repulsions.} Unlike the Lupu's loop soup, the clusters for the critical random walk loop soup is strictly dominated by the sign clusters of the GFF on the metric graph. In order to address this, we apply the aforementioned random current model and see that the loop clusters dominates a generalized sign cluster on the metric graph, where we replace each original edge (which can be viewed as a unit resistor) by two edges and assign the conductances so that it sums to 1. This is summarized in Lemma~\ref{coupling}. When employing the proof idea of Theorem~\ref{thm-loop-soup-critical}, we encounter a problem which amounts to bounding the typical value of a GFF under the conditioning of staying positive in a subset. Results of this type, on such entropic repulsions for two-dimensional GFFs under the presence of hard wall, has been obtained in \cite{Dunlop1992,BDG01}. Our set up is slightly more complicated (and somewhat non-standard), and dealing with it forms the main technical ingredient in Section~\ref{sec-random-walk-loop-soup}. As standard in this type of problems, our proof crucially relies on the FKG inequality \cite{FKG71,Preston74} and the Brascamp-Lieb inequality \cite{BL76}.

\subsection{Open problems}

Our results motivate a number of interesting questions, as we list below.

\begin{question}
	For the random walk loop soup in the supercritical regime (i.e., with intensity strictly larger than $\frac{1}{2}$), is the dimension of the chemical distance 1 with high probability?
\end{question}

\begin{question}
	Can one prove an analogous result for Brownian loop soups?
\end{question}

Next, we will ask a number of questions in the context of level set percolation for GFF, but one can ask natural analogous questions for loop soups as well as random walks. We feel that, perhaps the questions regarding to GFF may be answered before that on random walks and loop soups.

\begin{question}
	Under assumptions of Theorem~\ref{main-thm-annulus}, is the dimension of chemical distance 1 with high probability conditioned on the existence of an open crossing?
\end{question}

\begin{question}\label{question-super-linear}
	Under assumptions of Theorem~\ref{main-thm-annulus}, is the length of minimal open crossing $O(n)$ with positive probability?
\end{question}

\begin{question}
	Under assumptions of Theorem~\ref{main-thm-annulus}, is the number of disjoint open crossings tight?
\end{question}

Finally, we pose a question regarding to universality of Theorem~\ref{main-thm-annulus}, whose difficulty is due to the crucial role of Makarov's theorem (which seems to only apply for GFF) in the proof of Theorem~\ref{main-thm-annulus}. In fact, we choose to keep an open mind on whether such universality holds, in light of a non-universality result in \cite{DZ15} on a different  type of metric related to GFF.
\begin{question}
	Does an analogous result to Theorem~\ref{main-thm-annulus} hold for all log-correlated Gaussian fields?
\end{question}

\section{Percolation for Gaussian free fields}

This section is devoted to the proof of Theorem~\ref{main-thm-annulus}.
For notation convenience, we say a vertex $v$ is $\lambda$-open (or open if no risk of confusion) if $v\in \mathcal H_{N, \lambda}$, and $\lambda$-closed (or closed) otherwise.  For any $A, B \subseteq V_N$, we denote by $A \overset{\leq \lambda}{\longleftrightarrow} B$ the event that there exists a $\lambda$-open path $P$ connecting $A$ and $B$, i.e., $D_{N, \lambda}(A, B)<\infty$.

\subsection{One-arm estimate: a warm up argument} \label{sec-warm-up}
In this subsection, we give a warm up argument on level set percolation for GFF. Despite being rather simple, the argument is a clear demonstration of the fundamental idea of the paper, which allows to take advantage of the Markov field property of GFF in studying  percolations.  We remark that a similar argument was employed in \cite[Section 3]{Sznitman15}.

\begin{prop}\label{prop-warm-up}
	For any $0<\alpha<\beta < 1$, there exists a constant $c>0$ such that for all $\lambda > 0$
	$$\P(\partial V_{\alpha N} \overset {\leq \lambda}{\longleftrightarrow} \partial V_{\beta N}) \geq 1 - 2\mathrm{e}^{-c\lambda^2}\,.$$
\end{prop}
In order to prove Proposition~\ref{prop-warm-up}, we will need  the following standard estimates on simple random walks; we include a proof merely for completeness. 
	\begin{lemma}\label{random-walk-lemma-1}
		For any fixed $0< r<1$, there exist constants $c_1, c_2>0$ which depend on $r$ such that
		\begin{equation}\label{random-walk-lemma-1-eq-1}
			\sum\limits_{v\in \partial V_{rN}} G_{V_N}(u,v)\leq c_1N, \quad \forall u\in \partial V_{rN}
		\end{equation}
		and
		\begin{equation}\label{random-walk-lemma-1-eq-2}
			G_{V_N}(u,v)\geq c_2, \quad \forall u,v\in V_{rN}\,.
		\end{equation}
		Furthermore, for any $0<\alpha< \beta<1$, there exists a constant $c_3>0$ such that for all $u\in \partial V_{\beta N}$, the simple random walk started at $u$ will hit $\partial V_{\alpha N}$ before $\partial V_N$ with probability at least $c_3$.
	\end{lemma}
	
	\begin{proof}
Let $S_n=(S_{1, n}, S_{2, n})$ be a simple random walk on $\mathbb Z^2$. It is clear that if $S$ is on $\partial V_{rN}$ at some point, in the next step it will move to some vertex on $\partial V_{rN+2}$ with probability at least 1/4 (note that $\partial V_{rN}$ and $\partial V_{rN+2}$ are two neighboring boundaries), and after that, it will hit $\partial V_N$ before $\partial V_{rN}$ with probability at least $\frac 1 {(1-r)N}$ (since $\max\{|S_{1, n}|,|S_{2, n}|\}$ is a submartingale). Therefore, a simple random walk started at any $u\in \partial V_{rN}$ will in expectation visit $\partial V_{rN}$ at most $4(1-r)N$ times before hitting $\partial V_N$. This proves our first bound \eqref{random-walk-lemma-1-eq-1}.
		
		For the second bound \eqref{random-walk-lemma-1-eq-2}, let $\epsilon = \frac{1-r}{100}$. Denote by $u = (u_1, u_2)$ and $v=(v_1, v_2)$. By independence of the simple random walks in $x$ and $y$-coordinates, there exists $c' = c'(r)$ such that with probability at least $c'$ the simple random walk started at $u$ will hit some point $v^*$ in the vertical line $x = v_1$ before exiting $V_N$ or before the $y$-coordinate deviates by more than $\epsilon N$; started from $v^*$, there is again probability at least $c'$ for the simple random walk to hit the horizontal line $y=v_2$ before the horizontal coordinate deviates by more than $\epsilon N$. Altogether, there is probability at least $(c')^2$ for the random walk to hit the $\ell_\infty$-ball of radius $\epsilon N$ around $v$ before exiting $V_N$. At this point, an application of \cite[Proposition 4.6.2, Theorem 4.4.4.]{LL10} completes the verification of \eqref{random-walk-lemma-1-eq-2}.
		
		The last statement of lemma was implicitly proved in the above derivation of \eqref{random-walk-lemma-1-eq-2}.\qed 
	\end{proof}

\begin{proof}[Proof of Proposition~\ref{prop-warm-up}]
We say two vertices $u$ and $v$ are $*$-connected if the $\ell_\infty$-norm of $u-v$ is 1, and we call a $*$-connected cycle as a contour.
	By planar  duality,  the complement of the event $\{\partial V_{\alpha N} \overset {\leq \lambda}{\longleftrightarrow} \partial V_{\beta N}\}$ is the same as the event that there exists a $\lambda$-closed contour $\mathcal C \subseteq V_{\beta N}$ surrounding (each vertex of) $V_{\alpha N}$ (we say a vertex $v$ is surrounded by $\mathcal C$ if any path from $v$ to $\partial V_{\beta N}$ has to intersect with $\mathcal C$). We let $\mathfrak C$ be the collection of all such contours. It suffices to estimate $\P(\mathfrak C \neq \emptyset)$. 
	
	To this end, we consider a natural partial order on all contours. For any contour $\mathcal C$, we let $\bar {\mathcal C}$ be the collection of vertices that are surrounded by $\mathcal C$. For two contours $\mathcal C_1$ and $\mathcal C_2$, we say $\mathcal C_1 \leq \mathcal C_2$ if $\bar {\mathcal C_1} \subseteq \bar {\mathcal C_2}$. A key observation is that this partial order generates a well-defined (unique) global minimum on $\mathfrak C$, which we denote by $\mathcal C^*$. Furthermore, for any contour $\mathcal C \subseteq V_{\beta N}$ surrounding $V_{\alpha N}$, we have
	\begin{equation}\label{eq-stopping-contour}
		\{ \mathcal C^* = \mathcal C \} \in \mathcal F_{ \bar{\mathcal C}} \deq \sigma(\{\eta_{N, v}: v\in  \bar{\mathcal C}\})\,.
	\end{equation} 
	
	Define our ``observable'' $X$ to be
	\begin{equation}\label{eq-define-X}
		X = \frac{1}{|\partial V_{(1+\beta)N/2}|}\sum_{v\in \partial V_{(1+\beta)N/2}} \eta_{N, v}\,.
	\end{equation}
		As a simple corollary of \eqref{random-walk-lemma-1-eq-1}, 	there exists a constant $c_4>0$ which depends on $\beta$ such that
	\begin{equation}\label{Var-X}
		\var X=\frac{1}{4}\frac{1}{|\partial V_{(1+\beta)N/2}|^2}\sum_{u,v\in \partial V_{(1+\beta)N/2}} G_{V_N}(u,v)\leq c_4
	\end{equation}
	and thus we also have $\var (X \mid \mathcal F_{\bar{\mathcal C}}) \leq c_4$.
	
	By the Markov field property of the GFF, we have for each $v\in \partial V_{(1+\beta)N/2}$
	\begin{equation}\label{eq-conditional-expectation-given-Contour-1}
		\E (\eta_{N,v} \mid \mathcal F_{\bar{\mathcal C}}) = \sum_{u\in \mathcal C} \mathrm{Hm}(v, u; \mathcal C \cup \partial V_N) \cdot  \eta_{N, u} \,.
	\end{equation} 
	Recall that for a set $A$, we use $\mathrm{Hm}(v, u; A)$ to denote the harmonic measure at $u$ with respect to starting point $v$ and the target set $A$ (i.e., $\mathrm{Hm}(v, u; A) = \P_v( S_{\tau_A} = u)$, where $(S_n)$ is a simple random walk on $\mathbb Z^2$ and $\tau_A$ is the first time it hits set $A$). Also recall that $\mathrm{Hm}(v, B; A) = \sum_{u\in B} \mathrm{Hm}(v, u; A)$. Now on the event $\{\mathcal C^* = \mathcal C\}$ , we have $\eta_{N, u}\geq \lambda$ for all $u\in \mathcal C$. Combined with Lemma~\ref{random-walk-lemma-1}, it gives that
	\begin{equation}\label{eq-conditional-expectation-given-Contour}
		\E (\eta_{N,v} \mid \mathcal F_{\bar{\mathcal C}})  \geq \lambda \mathrm{Hm}(v, \mathcal C; \mathcal C \cup \partial V_N) \geq \lambda \mathrm{Hm}(v, \partial V_{\alpha N}; \partial V_{\alpha N} \cup \partial V_N) \geq c_3 \lambda\,.
	\end{equation}
	Therefore, we have $\E (X \mid \mathcal F_{\bar{\mathcal C}}) \geq c_3\lambda$ on the event $\{\mathcal C^* = \mathcal C\}$. Thus, 
	$$\P(X\geq c_3\lambda/2 \mid \mathcal F_{\bar{\mathcal C}})\geq 1-\P(Z(c_4)\geq c_3\lambda/2) \text { on the event } \{\mathcal C^* = \mathcal C\}\,,$$
	where we $Z(c_4)$ is a mean zero Gaussian variable with variance $c_4$.
	Since $\{ \mathcal C^* = \mathcal C \} \in \mathcal F_{ \bar{\mathcal C}}$, we have
	$$\P(X \geq c_3 \lambda /2 \mid \mathcal C^* = \mathcal C) \geq 1-\P(Z(c_4)\geq c_3\lambda/2)\,.$$
	Summing this over all possible contours $\mathcal C \subseteq V_{\beta N}$ surrounding $V_{\alpha N}$, we obtain that
	$$\P(X \geq c_3 \lambda/2 \mid \mathfrak C \neq \emptyset) \geq 1-\P(Z(c_4)\geq c_3\lambda/2)\,.$$
	Combined with the simple fact that 
	\begin{equation}\label{eq-X-tail}
		\P(X \geq c_3\lambda/2) \leq \P(Z(c_4)\geq c_3\lambda/2)\,,
	\end{equation}
	it follows that
	$$\P(\mathfrak C \neq \emptyset) = \frac{\P(X \geq c_3\lambda/2)}{\P(X \geq c_3\lambda/2 \mid \mathfrak C \neq \emptyset) } \leq \frac {\P(Z(c_4)\geq c_3\lambda/2)}{1-\P(Z(c_4)\geq c_3\lambda/2)}\,.$$
	This completes the proof of the proposition.\qed
\end{proof}

\subsection{Proof of Theorem~\ref{main-thm-annulus}}\label{sec:annulus} The proof of Theorem~\ref{main-thm-annulus} is inspired from the proof of Proposition~\ref{prop-warm-up} but with important difference: in Proposition~\ref{prop-warm-up} we work with a set $\bar {\mathcal C}$ which is surrounded completely by a $\lambda$-closed contour; in the present case, we will instead work with a set that is surrounded by a contour which is $\lambda$-closed \emph{except for a small fraction of vertices} --- the harmonic measure on this small fraction of $\lambda$-open vertices 
is then controlled by Theorem~\ref{thm-Makarov}.  We encapsulate the consequence of Theorem~\ref{thm-Makarov} in the following lemma which suits for applications in the present article.
\begin{lemma}\label{lem-Makarov}
For any $0<\alpha<\beta < 1$ and $\chi>1/2$, the following holds for all connected set $\mathsf C \subseteq V_{\beta N}$ with diameter at least  $\alpha N$, for all $\mathsf A\subseteq \mathsf C$ with $|\mathsf A| \leq N \mathrm{e}^{-(\log N)^{\chi}}$, and for all $v\in \partial V_{(1+\beta) N/2}$:
\begin{equation}\label{eq-A_k-small}
\mathrm{Hm}(v, \mathsf A; \mathsf C\cup \partial V_N)=o(\log N)^{-10}\,.
\end{equation}
\end{lemma}
\begin{proof}
 First, we note that 
\begin{equation}\label{eq-A_k-1}
	\mathrm{Hm}(v, \mathsf A; \mathsf C \cup \partial V_N) \leq \mathrm{Hm}(v, \mathsf A; \mathsf C)\,.
\end{equation}
By a combination of Theorem 1.7.6 (Harnack principle), Theorem 2.1.3 and Exercise 2.1.4 in \cite{Lawler91}, we have for constants $c_6, c_7,c_8>0$ which depend on $\beta$, any $u\in \mathsf C$ and arbitrary $w\in \partial V_{20N}$
$$\mathrm{Hm}(v, u;\mathsf C)\leq c_6 \mathrm{Hm}(w, u; \mathsf C)\leq c_7 \mathrm{Hm} (8N,u; \mathsf C)\leq c_8\mathrm{Hm}(\infty, u;\mathsf C)\,,$$
where $\mathrm{Hm} (8N,u; \mathsf C)$ corresponds to the $H_A^m(y)$ in \cite[Theorem 2.1.3]{Lawler91} with $ A=\mathsf C$, $m=8N$ and $y=u$ (where $A, m, y$ are notations in \cite{Lawler91}). 
Therefore,
\begin{equation}\label{eq-A_k-2}
	\mathrm{Hm}(v, \mathsf A; \mathsf C)\leq c_8\mathrm{Hm}(\infty, \mathsf A; \mathsf C)\,.
\end{equation}
Choose $\chi'$ such that $1/2<\chi'< \chi$. Since $\mathsf C$ is a connected set of radius between $(\alpha/2)N$ and $2N$, by Theorem~\ref{thm-Makarov} we deduce that for constants $c_{9}, c_{10}>0$ depending only on $\alpha$ and $\chi'$
$$\mathrm{Hm}(\infty, \{u\in \mathsf C: \mathrm{Hm}(\infty, u; \mathsf C) > c_{9}N^{-1} \mathrm{e}^{(\log N)^{\chi'}}\}; \mathsf C)\leq c_{10} (\log N)^{-20}\,.$$
Therefore,
\begin{eqnarray}\label{eq-A_k-3}
	\mathrm{Hm}(\infty, \mathsf A; \mathsf C)&=&\mathrm{Hm}(\infty, \mathsf A\cap\{u\in \mathsf C: \mathrm{Hm}(\infty, u; \mathsf C) \leq c_{9}N^{-1} \mathrm{e}^{(\log N)^{\chi'}}\}; \mathsf C)\nonumber\\
	&&+\mathrm{Hm}(\infty, \mathsf A\cap\{u\in \mathsf C: \mathrm{Hm}(\infty, u; \mathsf C) > c_{9}N^{-1} \mathrm{e}^{(\log N)^{\chi'}}\}; \mathsf C)\nonumber\\
	&\leq& (c_{9}N^{-1} \mathrm{e}^{(\log N)^{\chi'}})\cdot N \mathrm{e}^{-(\log N)^{\chi}}+c_{10} (\log N)^{-20}\nonumber\\
	&=& o(\log N)^{-10}\,.
\end{eqnarray}
Combining \eqref{eq-A_k-1}, \eqref{eq-A_k-2} and \eqref{eq-A_k-3}, we finally conclude \eqref{eq-A_k-small}, completing the proof of the lemma. \qed
\end{proof}

\begin{proof}[Proof of Theorem~\ref{main-thm-annulus}]
In what follows, we implement the proof of Theorem~\ref{main-thm-annulus} in three steps.

\noindent {\bf Step 1: construct an almost closed surrounding contour.} Consider $\lambda > 0$. Our goal is to provide a lower bound on the probability that there exists a $\lambda$-open path with length less than $N \mathrm{e}^{(\log N)^{\chi}}$ connecting $\partial V_{\alpha N}$ and $\partial V_{\beta N}$ for some $\chi>1/2$. Note that the distance between $\partial V_{\alpha N}$ and $\partial V_{\beta N}$ is the same as the distance between $V_{\partial N}$ and $\partial V_{\beta N}$.  This motivates the following definitions for $i\geq 1$:
\begin{equation}\label{eq-mathcal-A-B-C}
\begin{split}
\mathcal A_i & = \{v\in V_{N}\cap \mathcal H_{N, \lambda}: D_{N, \lambda} (V_{\alpha N}\cap \mathcal H_{N, \lambda}, v) = i\}\,, \\
\mathcal B_{i}&= \{v\in V_{N} \setminus (\mathcal H_{N, \lambda} \cup V_{\alpha N}): v\sim u \mbox{ for some } u\in \mathcal A_{j} \mbox{ and } 1\leq j < i\} \cup (\partial V_{\alpha N}\setminus \mathcal H_{N, \lambda}) \,,\\
\mathcal I_{i}&= (\cup_{j=1}^{i} \mathcal A_j) \cup \mathcal B_{i} \cup V_{\alpha N}\,.
\end{split}
\end{equation}
It would be beneficial to picture the preceding definitions obtained from the following exploration process. We set $\mathcal A_0=V_{\alpha N}\cap \mathcal H_{N, \lambda}$, $\mathcal B_0=\partial V_{\alpha N}\setminus \mathcal H_{N, \lambda}$, $\mathcal I_0= V_{\alpha N}$, and for $i=0, 1, 2, \ldots$, we see that inductively
\begin{align*}
	\mathcal A_{i+1}& = \{v\in (V_{N} \setminus \mathcal I_i)\cap \mathcal H_{N, \lambda}: v\sim u \mbox{ for some } u\in \mathcal A_i\}\,,\\
	\mathcal B_{i+1}&= \{v\in (V_{N} \setminus \mathcal I_i)\setminus \mathcal  H_{N, \lambda}: v\sim u \mbox{ for some } u\in \mathcal A_i\} \cup \mathcal B_i\,,\\
	\mathcal I_{i+1}&= \mathcal I_i \cup \mathcal A_{i+1}\cup \mathcal B_{i+1}\,.
\end{align*}
In other words, we can think of constructing the sets $\mathcal A_{i+1}, \mathcal B_{i+1}, \mathcal I_{i+1}$ for $i\geq 0$ using the following procedure. At stage $i+1$, we explore all the neighbors of $\mathcal A_i$ that is in $V_{N}\setminus \mathcal I_i$ (that is, vertices which have not been explored): if the vertex is in $\mathcal H_{N, \lambda}$ then we put it to $\mathcal A_{i+1}$, otherwise we put it to $\mathcal B_{i + 1}$. For $i\geq 1$ it is clear that $\mathcal A_i$ records all the vertices in $V_{N}$ that are of chemical distance $i$ to $V_{\alpha N}$ (i.e., all vertices in $V_{N}\setminus V_{\alpha N}$ that are of chemical distance $i$ to $\partial V_{\alpha N}$), $\mathcal B_i$ records all the closed vertices we have encountered, and  $\mathcal I_i$ records all the vertices that have been explored before (or at) stage $i$ --- thus, this construction from the exploration algorithm is indeed consistent with definitions given in \eqref{eq-mathcal-A-B-C}. Furthermore, we observe that the following hold for all $i\geq 1$ as long as $\mathcal I_i \cap \partial V_N = \emptyset$:
\begin{itemize}
	\item  $\mathcal A_i$'s are disjoint from each other.
	\item  $\mathcal I_i$ is a connected set in $V_{N}$.
	\item  $\partial \mathcal I_i$ (i.e, $\{w\in \mathcal I_i: w'\sim w \mbox{ for some } w' \not\in \mathcal I_i\}$) is a subset of $\mathcal A_i\cup \mathcal B_i$.
	\item Let $\mathcal C_i\deq\{u: \mathrm{Hm}(\infty, u;\mathcal I_i)>0\}$ (note that by definition $\mathcal C_i$ is a surrounding contour). Then $\mathcal C_i\subseteq \partial \mathcal I_i  \subseteq \mathcal A_i\cup \mathcal B_i$.
\end{itemize}

Now suppose that the event $\mathcal E\deq\{D_{N, \lambda}(\partial V_{\alpha N}, \partial V_{\beta N}) \geq N \mathrm{e}^{(\log N)^{\chi}}\}$ occurs, then we must have that $\mathcal I_i$ is disjoint from $V_N\setminus V_{\beta N}$ for all $0\leq i< N \mathrm{e}^{(\log N)^{\chi}}$. Further, since $\mathcal A_i$'s are disjoint from each other, we see (from a simple volume consideration) that there exists at least an $i_0 < N \mathrm{e}^{(\log N)^{\chi}}$ such that 
\begin{equation}\label{eq-i-0}
	|\mathcal A_{i_0}| \leq N \mathrm{e}^{-(\log N)^{\chi}}\,.
\end{equation}
We let $\tau$ be the minimal number $i_0$ which satisfies \eqref{eq-i-0}. In summary, we have
\begin{equation}\label{disjoint-union}
	\mathcal E\subseteq \mathcal E'\deq \bigcup_{\substack{0\leq k< N \mathrm{e}^{(\log N)^{\chi}}\\ (A_0, \ldots A_k, B_0,\ldots B_k)\in \mathcal P_k}} \{\tau = k\}\cap \{\mathcal A_i = A_i, \mathcal B_i = B_i \mbox{ for } 0\leq i\leq k\}
\end{equation}
where $\mathcal P_k$ contains all $(A_0, \ldots A_k, B_0,\ldots B_k)$ such that  $\{\tau = k\}\cap \{\mathcal A_i = A_i, \mathcal B_i = B_i \mbox{ for } 0\leq i\leq k\} \neq \emptyset$. In particular, they satisfy the following properties:
\begin{itemize}
	\item Denote $I_k\deq \cup_{i=0}^k A_i \cup B_k \cup V_{\alpha N}$. Then $I_k$ is a connected set in $V_{\beta N}$.
	\item Denote $C_k\deq\{u: \mathrm{Hm}(\infty, u; I_k)>0\}$. Then $C_k\subseteq A_k\cup B_k$.
	\item $|A_k|\leq N \mathrm{e}^{-(\log N)^{\chi}}$.
\end{itemize}

Now we fix any $0\leq k< N \mathrm{e}^{(\log N)^{\chi}}$ and any $(A_0, \ldots A_k, B_0,\ldots B_k)\in \mathcal P_k$. It is not hard to verify that (this is indeed obvious in light of the  construction of $\mathcal A_i, \mathcal B_i$ and $\mathcal I_i$ via exploration process)
\begin{equation}\label{C_k-measurable}
	\{\tau = k\}\cap \{\mathcal A_i = A_i, \mathcal B_i = B_i \mbox{ for } 0\leq i\leq k\}\in \mathcal F_{I_k}\,.
\end{equation}
\noindent {\bf Step 2: control the harmonic measure.} Since $V_{\alpha N} \subseteq I_k\subseteq V_{\beta N}$, we have from Lemma~\ref{random-walk-lemma-1} that  for each $v\in \partial V_{(1+\beta) N/2}$
\begin{equation}\label{eq-C_k-1}
	\mathrm{Hm}(v, I_k; I_k \cup \partial V_N) \geq \mathrm{Hm}(v, \partial V_{\alpha N}; \partial V_{\alpha N} \cup \partial V_N) \geq c_3\,.
\end{equation}
In addition, we have that
\begin{equation}\label{eq-C_k-2}
	\{u\in I_k: \mathrm{Hm}(v, u;I_k\cup \partial V_N)>0\}=\{u\in I_k: \mathrm{Hm}(\infty, u;I_k)>0\}=C_k \subseteq A_k\cup B_k\,.
\end{equation}
By an application of  Lemma~\ref{lem-Makarov},  we get that
\begin{equation}\label{eq-Makarov-Section-2}
\mathrm{Hm}(v, A_k; I_k\cup \partial V_N)=o(\log N)^{-10}\,.
\end{equation}
Combined \eqref{eq-C_k-1} and \eqref{eq-C_k-2}, it yields that
\begin{equation}\label{eq-B-k-harmonic-measure}
\mathrm{Hm}(v, B_k; I_k\cup \partial V_N)\geq c_3-o(\log N)^{-10}\,.
\end{equation}

\noindent {\bf Step 3: use Markov field property of GFF.} Conditioned on $\mathcal F_{I_k}$, the field $\{\eta_{N, v}: v\in V_N \setminus I_k\}$ is again distributed as a GFF. In particular, for each $v\in \partial V_{(1+\beta) N/2}$, we have
\begin{equation}\label{eq-expression-conditional-expectation}
	\E(\eta_{N, v} \mid \mathcal F_{I_k}) = \sum_{u\in I_k} \mathrm{Hm}(v, u; I_k \cup \partial V_N) \cdot \eta_{N, u}\,.
\end{equation}
 Now on the event $\{\tau = k\}\cap \{\mathcal A_i = A_i, \mathcal B_i = B_i \mbox{ for } 0\leq i\leq k\}$, we have by definition that $\eta_{N, u} \geq \lambda$ for all $u\in B_k$, so we can derive from \eqref{eq-expression-conditional-expectation} that
\begin{align*}
	\E(\eta_{N, v} \mid \mathcal F_{I_k}) &= \sum_{u\in I_k} \mathrm{Hm}(v, u; I_k \cup \partial V_N) \cdot \eta_{N, u}\\
	&= \sum_{u\in A_k} \mathrm{Hm}(v, u; I_k \cup \partial V_N) \cdot \eta_{N, u}+ \sum_{u\in B_k} \mathrm{Hm}(v, u; I_k \cup \partial V_N) \cdot \eta_{N, u}\\
	&\geq (c_3-o((\log N)^{-10}))\lambda - o((\log N)^{-10})\sup_{u\in V_N} |\eta_{N, u}|\,,
\end{align*}
where in the last inequality we used \eqref{eq-Makarov-Section-2} and \eqref{eq-B-k-harmonic-measure}.
Define $\Lambda_{\mathrm{bad}} = \{\sup_{u\in V_N} |\eta_{N, u}| \geq 100 \log N\}$. By a straightforward computation, we have
\begin{equation}\label{eq-bad-event}
	\P(\Lambda_{\mathrm{bad}}) \leq N^{-20}\,.
\end{equation}
We can assume without loss that $\Lambda_{\mathrm{bad}}$ does not occur. To be precise, for sufficiently large $N$, on the event $\{\tau = k\}\cap \{\mathcal A_i = A_i, \mathcal B_i = B_i \mbox{ for } 0\leq i\leq k\}\setminus \Lambda_{\mathrm{bad}}$, we have
$$\E(\eta_{N, v} \mid \mathcal F_{I_k}) \geq 9c_3\lambda/10\,.$$
Recall the definition of $X$ in \eqref{eq-define-X}. Then on the same event, we have $\E (X \mid \mathcal F_{I_k}) \geq 9c_3\lambda /10$ and $\var (X \mid \mathcal F_{I_k}) \leq \var X\leq c_4$. Thus (still on the same event), 
$$\P(X\geq c_3\lambda/2 \mid \mathcal F_{I_k})\geq 1-\P(Z(c_4)\geq 2c_3\lambda/5)\,,$$
where $Z(c_4)$ is a mean zero Gaussian variable with variance $c_4$.
By \eqref{C_k-measurable}, this gives
\begin{eqnarray*}
	&&\P(X \geq c_3 \lambda /2, \tau = k, \mathcal A_i = A_i, \mathcal B_i = B_i \mbox{ for } 0\leq i\leq k)\\
	&=&\E (\P(X\geq c_3\lambda/2 \mid \mathcal F_{I_k})\1_{\{\tau = k\}\cap \{\mathcal A_i = A_i, \mathcal B_i = B_i \mbox{ for } 0\leq i\leq k\}})\\
	&\geq&(1-\P(Z(c_4)\geq 2c_3\lambda/5))\P(\{\tau = k\}\cap \{\mathcal A_i = A_i, \mathcal B_i = B_i \mbox{ for } 0\leq i\leq k\}\setminus \Lambda_{\mathrm{bad}})\,.
\end{eqnarray*}
Summing this over all $0\leq k< N \mathrm{e}^{(\log N)^{\chi}}$ and all $(A_0, \ldots A_k, B_0,\ldots B_k)\in \mathcal P_k$ and using \eqref{disjoint-union}, we have
$$\P(X \geq c_3 \lambda /2)\geq (1-\P(Z(c_4)\geq 2c_3\lambda/5)) \P(\mathcal E'\setminus \Lambda_{\mathrm{bad}})\,.$$
Therefore,
\begin{align*}
	\P(\mathcal E)\leq\P(\mathcal E')&\leq \P(\mathcal E'\setminus \Lambda_{\mathrm{bad}})+\P(\Lambda_{\mathrm{bad}})\\
	&\leq \frac{\P(X \geq c_3\lambda/2)}{1-\P(Z(c_4)\geq 2c_3\lambda/5) }+\P(\Lambda_{\mathrm{bad}})\\
	&\leq \frac {\P(Z(c_4)\geq c_3\lambda/2)}{1-\P(Z(c_4)\geq 2c_3\lambda/5)}+\P(\Lambda_{\mathrm{bad}})\,.
\end{align*}
Combined with \eqref{eq-bad-event}, this completes the proof of Theorem~\ref{main-thm-annulus}. \qed
\end{proof}

\section{Percolation of the continuous loop soup}\label{sec-continuous-loop}

In this section we prove an analogous result to Theorem~\ref{thm-random-walk-loop-soup-critical} for the continuous loop soups defined on the \emph{metric graph} of $\mathcal G_N = (V_N, E_N)$ at critical intensity $1/2$. The result in this section will not be used in the derivation of Theorem~\ref{thm-random-walk-loop-soup-critical} in Section~\ref{sec-random-walk-loop-soup}. However, our proof method of Theorem~\ref{thm-random-walk-loop-soup-critical} is hugely inspired by the consideration of the continuous loop soup. Therefore,  we include the present section, with the hope of conveying the source of insight.

The continuous loop soup as well as the Gaussian free field on the metric graph were considered in \cite{Lupu14}. We follow the setup and definitions there. We let $\tilde {\mathcal G}_N$ be the metric graph (or the cable system) of $\mathcal G_N$ where each edge in $\tilde {\mathcal G}_N$ has length $\frac 1 2$. On $\tilde {\mathcal G}_N$ we can define a standard Brownian motion $B^{\tilde {\mathcal G}_N}$, so that $B^{\tilde {\mathcal G}_N}$ when restricted to $V_N$ is the same as the aforementioned continuous-time sub-Markovian jump process $(X_t)$. Let $G_{\tilde {\mathcal G}_N}(u,v)$ be the Green's function of $B^{\tilde {\mathcal G}_N}$, so that  for $u,v\in V_N$, $G_{\tilde {\mathcal G}_N}(u,v)=\frac 1 4 G_{V_N}(u,v)$, the Green's function of $(X_t)$. Let $\{\tilde {\eta}_{N, v}:v\in \tilde {\mathcal G}_N\}$ be the Gaussian free field on $\tilde {\mathcal G}_N$ with covariance function $ G_{\tilde {\mathcal G}_N}(u,v)$. Then the restriction of $\{\tilde {\eta}_{N, v}:v\in \tilde {\mathcal G}_N\}$ to $V_N$ is the same as the Gaussian free field $\{ {\eta}_{N, v}:v\in V_N\}$. Moreover, $\{\tilde {\eta}_{N, v}:v\in \tilde {\mathcal G}_N\}$ can be obtained from $\{ {\eta}_{N, v}:v\in V_N\}$ by (for each edge $e = (u, v)$) independently sampling a variance 2 Brownian bridge of length $\frac 1 2$ with values $\eta_{N, u}$ and $\eta_{N, v}$ at the endpoints. In particular, as shown in \cite{Lupu14} $\{\tilde {\eta}_{N, v}:v\in \tilde {\mathcal G}_N\}$ has a continuous realization.

Now we can associate to $B^{\tilde {\mathcal G}_N}$ a measure $\tilde \mu_N$ on continuous loops in $\tilde {\mathcal G}_N$, and for each $\alpha>0$ consider the continuous loop soup $\tilde {\mathcal L}_{\alpha,N}$ which is a Poisson point process with intensity $\alpha \tilde \mu_N$. The loops of $\tilde {\mathcal L}_{\alpha,N}$ may be partitioned into clusters. For $u,v\in \tilde {\mathcal G}_N$, we define the chemical distance of $\tilde {\mathcal L}_{\alpha,N}$ between $u$ and $v$ by (below $|\gamma|$ is the length of the path $\gamma$) 
$$ D_{\tilde {\mathcal L}_{\alpha, N}} (u, v) = \min_{\gamma} |\gamma|\,,$$
where the minimum is over all path $\gamma\subseteq\tilde {\mathcal G}_N$ joining $u$ and $v$ that stays within a cluster of $\tilde {\mathcal L}_{\alpha,N}$.
\begin{theorem}\label{thm-loop-soup-critical}
	For any $0<\alpha<\beta < 1$ and $\chi>1/2$, there exists a constant $c>0$ such that for all $N$
	$$\P(D_{\tilde {\mathcal L}_{1/2, N}}(\partial V_{\alpha N}, \partial V_{\beta N}) \leq N \mathrm{e}^{(\log N)^{\chi}}) \geq c\,.$$
\end{theorem}
\begin{remark}\label{rem-3.2}
	We expect that  $\P(D_{\tilde {\mathcal L}_{1/2, N}}(\partial V_{\alpha N}, \partial V_{\beta N}) < \infty)$ is strictly less than 1; see Remark~\ref{rem-main}. 
\end{remark}

By \cite[Proposition 2.1]{Lupu14}, there is a coupling between $\tilde {\mathcal L}_{1/2, N}$ and a continuous version of $\{\tilde {\eta}_{N, v}:v\in \tilde {\mathcal G}_N\}$ such that the clusters of loops of $\tilde {\mathcal L}_{1/2, N}$ are exactly the sign clusters of $\{\tilde {\eta}_{N, v}:v\in \tilde {\mathcal G}_N\}$. In light of this, define for $u\neq v$
$$ D_{\tilde {\eta}_{N}} (u, v) = \min_{\gamma} |\gamma|\,,$$
where the minimum is over all path $\gamma\subseteq\tilde {\mathcal G}_N$ joining $u$ and $v$ such that $\tilde {\eta}_{N,\gamma}$ (including $u$ and $v$) are of the same sign (\emph{strictly} plus or minus). We continue to use the convention that $D_{\tilde {\eta}_{N}} (A, B) = \min_{u\in A, v\in B} D_{\tilde {\eta}_{N}}(u, v)$ for $A, B \subseteq \tilde {\mathcal G}_N$, and the convention that $\min \emptyset  = \infty$.
In order to prove Theorem~\ref{thm-loop-soup-critical}, it suffices to prove the following proposition. 
\begin{prop}\label{prop-continuous-loop-soup}
	For any $0<\alpha<\beta < 1$, $\chi>1/2$, there exists a constant $c>0$ such that for all $N$
	$$\P(D_{\tilde {\eta}_{N}}(\partial V_{\alpha N}, \partial V_{\beta N}) \leq N \mathrm{e}^{(\log N)^{\chi}}) \geq c\,.$$
\end{prop}
The following lemma will be useful in the proof of Proposition~\ref{prop-continuous-loop-soup}.
\begin{lemma}\label{random-walk-lemma-3}
		There exists a constant $c_{11}>0$ which depends on $\alpha$ and $\beta$ such that
		\begin{equation}\label{random-walk-lemma-3-eq}
			\frac{1}{4}\frac{1}{|\partial V_{(1+\beta)N/2}|^2}(\sum_{u,v\in \partial V_{(1+\beta)N/2}} G_{V_N}(u,v)-\sum_{u,v\in \partial V_{(1+\beta)N/2}} G_{V_N\setminus V_{\alpha N}}(u,v))\geq c_{11}\,.
		\end{equation}
\end{lemma}
\begin{proof}
		We consider two scenarios where in scenario (1) we kill the random walk upon hitting $\partial V_N$ and in scenario (2) we kill the random walk upon hitting $\partial V_{\alpha N} \cup \partial V_N$.  For any $u\in\partial V_{(1+\beta)N/2}$, we will compare the expected number of visits to $\partial V_{(1+\beta)N/2}$ of a simple random walk started at $u$ in these two scenarios. On the event $E$ that it hits $\partial V_N$ before $\partial V_{\alpha N}$, it will visit $\partial V_{(1+\beta)N/2}$ the same number of times in both scenarios (1) and (2). On the complement of $E$ (i.e. it hits $\partial V_{\alpha N}$ before $\partial V_N$), however, it will come back to some $w\in \partial V_{(1+\beta)N/2}$, and will then (conditionally in expectation) visit $\partial V_{(1+\beta)N/2}$ for exactly $\sum_{v\in \partial V_{(1+\beta)N/2}} G_{V_N}(w,v)$ more times in scenario (1) than in scenario (2). Since we have a uniform lower bound of $\sum_{v\in \partial V_{(1+\beta)N/2}} G_{V_N}(w,v)$ (by \eqref{random-walk-lemma-1-eq-2}) and that $\P(E^c)\geq c_3$ by Lemma~\ref{random-walk-lemma-1}, we see that 
		$$\sum_{v\in \partial V_{(1+\beta)N/2}} G_{V_N}(u,v)-\sum_{v\in \partial V_{(1+\beta)N/2}} G_{V_N\setminus V_{\alpha N}}(u,v)\geq c_3c_2|\partial V_{(1+\beta)N/2}|\,,$$ and \eqref{random-walk-lemma-3-eq} follows by summing this over all $u\in\partial V_{(1+\beta)N/2}$. \qed
\end{proof}
For a connected set $K\subseteq \tilde {\mathcal G}_N$, let $T_{K\cup\partial V_N}$ be the first time the Brownian motion $B^{\tilde {\mathcal G}_N}$ hits the closure of $K\cup\partial V_N$. It is clear that $B^{\tilde {\mathcal G}_N}_{T_{K\cup\partial V_N}}$ can only take  finitely many values $u$ in the closure of $K\cup\partial V_N$  regardless of the starting point of the Brownian motion. Thus, for any $v\in \tilde {\mathcal G}_N\setminus K$ and $u\in \tilde{\mathcal G}_N$ we can define 
	$$\widetilde {\mathrm{Hm}}(v, u; K\cup\partial V_N)\deq \P^v( B^{\tilde {\mathcal G}_N}_{T_{K\cup\partial V_N}} = u)$$ to be the harmonic measure of $B^{\tilde {\mathcal G}_N}$ at $u$ with respect to starting point $v$ and target set $K\cup\partial V_N$.	
\begin{proof}[Proof of Proposition~\ref{prop-continuous-loop-soup}]
	Our proof strategy is highly similar to that of Theorem~\ref{main-thm-annulus} (and thus the proof in the present case is presented with slightly less details). One difference in the present case is that our sign clusters are subgraphs on the metric graph $\tilde {\mathcal G}_N$ and thus the boundary points of our clusters are not necessarily lattice points.
Consequently, our definition of the analogue of \eqref{eq-mathcal-A-B-C} is slightly more complicated.  For $i\geq 1$ we define (note that $D_{\tilde {\eta}_{N}}(v, A) = 0$ if $v\in A$)
\begin{equation}\label{eq-mathcal-A-B-I}
\begin{split}
\tilde {\mathcal I}_i & = \{v\in \tilde {\mathcal G}_N : D_{\tilde {\eta}_{N}} (v,  \tilde {\mathcal G}_{\alpha N}) \leq i\}\,,\\
\tilde {\mathcal A}_i &= \{v\in \tilde {\mathcal G}_N : D_{\tilde {\eta}_{N}} (v, \tilde {\mathcal G}_{\alpha N}) = i\}\,,\\
\tilde {\mathcal B}_i &=  \partial \tilde {\mathcal I}_i \setminus \tilde {\mathcal A}_i\,.
\end{split}
\end{equation}
Here  $\partial \tilde {\mathcal I}_i$ is the collection of all points $v$ such that any neighborhood of $v$ in $\tilde {\mathcal G}_N$ has non-empty intersection with $\tilde {\mathcal I}_i$ but is not a subset of $\tilde {\mathcal I}_i$. We remark that $\tilde {\mathcal A}_i, \tilde {\mathcal B}_i, \tilde {\mathcal I}_i$ in \eqref{eq-mathcal-A-B-I} are analogues of $\mathcal A_i, \mathcal B_i, \mathcal I_i$ in \eqref{eq-mathcal-A-B-C} --- the tilde in the notation is to emphasize that  they are considered as subsets of the metric graph. We also note that for exposition convenience in \eqref{eq-mathcal-A-B-C} we have included $\mathcal B_i$ in $\mathcal I_i$, but in \eqref{eq-mathcal-A-B-I} $\tilde {\mathcal B_i}$ is on the boundary of $\tilde {\mathcal I}_i$ but is not a subset of  $\tilde {\mathcal I}_i$. 
 Note that any point in $\tilde {\mathcal A}_i$ is necessarily a lattice point, and that with probability 1 the GFF values at all lattice points  are non-zero (which we assume in what follows for clarify of exposition) and therefore $\tilde {\mathcal B}_i = \{v\in \partial \tilde {\mathcal I}_i: \tilde \eta_{N, v} = 0\}$ (from our definition, we see that the GFF values on $\tilde {\mathcal B}_i$ are necessarily zero). 
 One may consider an analogous exploration process construction as for \eqref{eq-mathcal-A-B-C}, but we omit the details. It is clear from the definition \eqref{eq-mathcal-A-B-I} that the following hold as long as $\tilde {\mathcal I}_i \cap \partial V_N = \emptyset$:
	\begin{itemize}
		\item  $\tilde {\mathcal I}_i$ is a connected set in $\tilde {\mathcal G}_N$. 
		\item  $\tilde {\mathcal A}_i$'s are disjoint from each other. In addition, $V_{\alpha N} \cup \cup_{j=1}^{i}\tilde {\mathcal A}_j$ is the set of all the lattice points in $\tilde {\mathcal I}_i$.
		\item  $\partial \tilde {\mathcal I}_i$ (the boundary points of $\tilde {\mathcal I}_i$) is a subset of $\tilde {\mathcal A}_i\cup \tilde {\mathcal B}_i$.
		\item For $\tilde {\mathcal C}_i\deq\{u:  \widetilde{\mathrm{Hm}}(\infty,u;\tilde {\mathcal I}_i)>0\}$ we have $\tilde {\mathcal C}_i\subseteq \partial \tilde {\mathcal I}_i \subseteq \tilde {\mathcal A}_i\cup\tilde {\mathcal B}_i$.
	\end{itemize}
	
	Suppose that the event $\mathcal E \deq \{D_{\tilde {\eta}_{N}}(\partial V_{\alpha N}, \partial V_{\beta N}) > N \mathrm{e}^{(\log N)^{\chi}}\}$ occurs, then $\tilde {\mathcal I}_i$ must be a subset of $\tilde {\mathcal G}_{\beta N}$ for all $0\leq i< N \mathrm{e}^{(\log N)^{\chi}}$. Further, let $\tau$ be the minimal number $i_0$ which satisfies
	$$|\tilde {\mathcal A}_{i_0}| \leq N \mathrm{e}^{-(\log N)^{\chi}}\,.$$
	Then since $\tilde {\mathcal A}_i$'s are disjoint from each other, we have $\tau< N \mathrm{e}^{(\log N)^{\chi}}$ on the event $\mathcal E$. Crucially, similar to \eqref{C_k-measurable}, we have that for all $k\geq 1$ and all $\{\tilde A_i, \tilde B_i, \tilde I_i: 1\leq i\leq k\}$
\begin{equation}\label{C_k-measurable-tilde}
	\{\tau = k\}\cap \{\tilde {\mathcal A}_i = \tilde A_i, \tilde {\mathcal B}_i = \tilde B_i, \tilde {\mathcal I}_i = \tilde I_i \mbox{ for } 0\leq i\leq k\}\in \mathcal F_{\tilde I_k}\,.
\end{equation}
In order to verify \eqref{C_k-measurable-tilde}, we can employ the exploration procedure similar to  the derivation for \eqref{C_k-measurable}. Since in Section~\ref{sec-random-walk-loop-soup} we will give a self-contained proof for a strictly stronger result than Theorem~\ref{thm-loop-soup-critical} (where we will also give a detailed description for the exploration procedure of slightly more complicated constructions), we omit the details on the verification of \eqref{C_k-measurable-tilde}.

Conditioned on $\mathcal F_\tau\deq \sigma(\{\tilde {\eta}_{N, v}:v\in \tilde {\mathcal I}_\tau\})$, by the strong Markov property in \cite[Section 3]{Lupu14}, $\{\tilde {\eta}_{N, v}:v\in \tilde {\mathcal G}_N\setminus \tilde {\mathcal I}_\tau\}$ is distributed as a mean zero GFF in $\tilde {\mathcal G}_N\setminus \tilde {\mathcal I}_\tau$ plus the harmonic extension of $\tilde {\eta}_{N, v}$ from $\tilde {\mathcal I}_\tau$ to $\tilde {\mathcal G}_N\setminus \tilde {\mathcal I}_\tau$. In particular, on the event $\mathcal E$ where $\tilde {\mathcal I}_\tau$ is contained in $\tilde {\mathcal G}_{\beta N}$, we have for each $v\in \partial V_{(1+\beta) N/2}$
	\begin{equation}\label{tilde-conditional-expectation}
		\E(\tilde\eta_{N, v} \mid \mathcal F_\tau) = \sum_{u\in \tilde {\mathcal I}_\tau} \widetilde {\mathrm{Hm}}(v, u; \tilde {\mathcal I}_\tau\cup\partial V_N) \cdot \tilde {\eta}_{N, u}=\sum_{u\in \tilde {\mathcal C}_\tau} \widetilde {\mathrm{Hm}}(v, u; \tilde {\mathcal I}_\tau\cup\partial V_N) \cdot \tilde {\eta}_{N, u}\,.
	\end{equation}
	Recall that by definition we have $\tilde {\eta}_{N, u}=0$ for all $u\in \tilde {\mathcal B}_\tau$. Thus
	\begin{equation}\label{tilde-Hm-Btau}
		\sum_{u\in \tilde {\mathcal B}_\tau} \widetilde {\mathrm{Hm}}(v, u; \tilde {\mathcal I}_\tau\cup\partial V_N) \cdot \tilde {\eta}_{N, u}=0\,.
	\end{equation}
	We want to show that $\widetilde {\mathrm{Hm}}(v, \tilde {\mathcal A}_\tau; \tilde {\mathcal I}_\tau\cup\partial V_N)$ is small. Let $\mathcal D_\tau= V_{\alpha N}\cup \cup_{j=1}^{\tau}\tilde {\mathcal A}_j$ be the set of all the lattice points in $\tilde {\mathcal I}_\tau$. Then $\mathcal D_\tau$ contains $\tilde {\mathcal A}_\tau$ and $V_{\alpha N}$ by definition, and it is a connected subgraph of $\mathcal G_{\beta N}$. Since $\mathcal D_\tau$ is a subset of  $\tilde {\mathcal I}_\tau$, and the restriction of (the Brownian motion) $B^{\tilde {\mathcal G}_N}$ on $V_N$ is the same as the simple random walk on $V_N$, we have
	\begin{equation}\label{tilde-Hm-Atau-1}
		\widetilde {\mathrm{Hm}}(v, \tilde {\mathcal A}_\tau; \tilde {\mathcal I}_\tau\cup\partial V_N) \leq \mathrm{Hm}(v, \tilde {\mathcal A}_\tau ; \mathcal D_\tau \cup \partial V_N)\,.
	\end{equation}
Furthermore, $|\tilde {\mathcal A}_\tau|\leq N \mathrm{e}^{-(\log N)^{\chi}}$. Therefore by Lemma~\ref{lem-Makarov}, we have
	\begin{equation}\label{tilde-Hm-Atau-2}
		\mathrm{Hm}(v,  \tilde {\mathcal A}_\tau; \mathcal D_\tau \cup \partial V_N)= o(\log N)^{-10}\,.
	\end{equation}
	
	Let $\Lambda_{\mathrm{bad}} = \{\sup_{u\in V_N} |\tilde {\eta}_{N, u}| \geq 100 \log N\}$ be as before. Then on the event $\mathcal E\setminus \Lambda_{\mathrm{bad}}$, combining \eqref{tilde-conditional-expectation}, \eqref{tilde-Hm-Btau}, \eqref{tilde-Hm-Atau-1} and \eqref{tilde-Hm-Atau-2} gives for each $v\in \partial V_{(1+\beta) N/2}$ that
	$$\E(\tilde\eta_{N, v} \mid \mathcal F_\tau)=o(\log N)^{-8}\,.$$
	Therefore (recall the definition of $X$ in \eqref{eq-define-X}),
	$$|\E(X \mid \mathcal F_\tau)|=o(\log N)^{-8}<\epsilon$$
	for some fixed $\epsilon>0$ and sufficiently large $N$.
	Now
	$$\var X = \frac 1 4\frac{1}{|\partial V_{(1+\beta)N/2}|^2}\sum_{u,v\in \partial V_{(1+\beta)N/2}} G_{V_N}(u,v)$$
	and on the event $\mathcal E$,
	\begin{align*}
		\var( X \mid \mathcal F_\tau) &=  \frac{1}{|\partial V_{(1+\beta)N/2}|^2}\sum_{u,v\in \partial V_{(1+\beta)N/2}} G_{\tilde {\mathcal G}_N\setminus \tilde {\mathcal I}_\tau}(u,v)\\
		&\leq \frac 1 4 \frac{1}{|\partial V_{(1+\beta)N/2}|^2}\sum_{u,v\in \partial V_{(1+\beta)N/2}} G_{V_N\setminus V_{\alpha N}}(u,v)\,.
	\end{align*}
	
		By Lemma~\ref{random-walk-lemma-3} we have $\var X-\var( X \mid \mathcal F_\tau)\geq c_{11}$ on the event $\mathcal E$. Also, recall that $\var X\leq c_4$ by \eqref{Var-X}.
	Now let $t=\epsilon+s {\sqrt{\var X-c_{11}}}$ where $s>0$ is a constant to be chosen later. Since given $\mathcal F_\tau$, $X$ is Gaussian, we have on the event $\mathcal E\setminus \Lambda_{\mathrm{bad}}$
	$$\P(X\leq t \mid \mathcal F_\tau) \geq \P(Z(\epsilon,\var X-c_{11})\leq t)=\P(Z\leq s)\,,$$
	where $Z(\epsilon, \var X - c_{11})$ is a Gaussian variable with mean $\epsilon$ and variance $\var X - c_{11}$, and $Z$ is a standard Gaussian variable. 
	In addition, we have
	$$\P(X \leq t)=\P(Z\leq \frac {\epsilon+s {\sqrt{\var X-c_{11}}}}{\sqrt{\var X}})\,.$$
	Since $c_{11}\leq\var X\leq c_4$, we have $0\leq \frac {\sqrt{\var X-c_{11}}}{\sqrt{\var X}}\leq \frac {\sqrt{c_4-c_{11}}}{\sqrt{c_4}}<1$, and thus
	$$\frac {\epsilon+s {\sqrt{\var X-c_{11}}}}{\sqrt{\var X}}\leq \frac {\epsilon}{\sqrt {c_{11}}}+s\frac {\sqrt{c_4-c_{11}}}{\sqrt{c_4}}\leq s(\frac {\frac {\sqrt{c_4-c_{11}}}{\sqrt{c_4}}+1}2)$$
	for a sufficiently large constant $s>0$.
	Therefore (recalling $\frac {\sqrt{c_4-c_{11}}}{\sqrt{c_4}}<1$)
	$$\P(\mathcal E\setminus \Lambda_{\mathrm{bad}})\leq \frac {\P(Z\leq \frac {\epsilon+s {\sqrt{\var X-c_{11}}}}{\sqrt{\var X}})}{\P(Z\leq s)}\leq \frac {\P(Z\leq s(\frac {\frac {\sqrt{c_4-c_{11}}}{\sqrt{c_4}}+1}2))}{\P(Z\leq s)}<1\,.$$
	Combined with \eqref{eq-bad-event}, this completes the proof of the proposition.\qed
\end{proof}

Finally, we remark that by an almost identical proof of Proposition~\ref{prop-continuous-loop-soup}, we can prove the next result.
\begin{cor}\label{cor-main-thm}
	For all $0<\alpha<\beta < 1$, $\chi>1/2$ and $\lambda>0$, there exists a constant $c>0$ such that
	$$\P(D_{N, -\lambda}(\partial V_{\alpha N}, \partial V_{\beta N}) \leq N \mathrm{e}^{(\log N)^{\chi}}) \geq c \mbox{ for all } N\geq 1\,.$$
\end{cor}

\section{Percolation of the random walk loop soup}\label{sec-random-walk-loop-soup}
This section is devoted to the proof of Theorem \ref{thm-random-walk-loop-soup-critical}, into which the three main proof ingredients (as discussed in Section~\ref{sec:proof-ingredients}) merge.
Recall that as stated in \eqref{eq-isomorphism}, the occupation time field $\{\hat {\mathcal L}^v_{1/2}\}_{v\in V_N}$ of $\mathcal L_{1/2, N}$ has the same law as $\{\frac{1}{2} \eta_{N, v}^2\}_{v\in V_N}$; and as stated around \eqref{eq-def-random-current}, conditioned on $\{\hat {\mathcal L}^v_{1/2}=\ell_v\}_{v\in V_N}$, the graph of $\mathcal L_{1/2, N}$ has the same law as $(1_{n_e>0})_{e\in E_N}$, where $(n_e)_{e\in E_N}$ follows the random current model (see \eqref{eq-def-random-current}) with parameters $\beta_e=2\sqrt{\ell_x\ell_y}$ on edge $e=(x,y)$.

Recall that each edge of ${\mathcal G}_N=(V_N,E_N)$ has conductance $1$. In this section we will consider a graph ${\mathcal G}'_N=(V_N, E_N(1)\cup E_N(2))$, where we replace each edge $e\in  E_N$ in the graph ${\mathcal G}_N$ by two multiple edges $e(1)$ and $e(2)$ and assign conductance $1/8$ to $e(1)$ and conductance $7/8$ to $e(2)$, and we denote $E_N(1)=\{e(1):e\in E_N\}$ and $E_N(2)=\{e(2):e\in E_N\}$. The graph ${\mathcal G}'_N$ is equivalent to ${\mathcal G}_N$ in the sense that the Gaussian free fields on ${\mathcal G}_N$ and ${\mathcal G}'_N$ have the same law.

As in \cite{Lupu14}, we will consider the Gaussian free field on the metric graph ${\tilde {\mathcal G}}'_N$ of ${\mathcal G}'_N$. The metric graph ${\tilde {\mathcal G}}'_N$ can be obtained from ${\mathcal G}'_N$ by assigning each edge $e(1)\in E_N(1)$ length $4$ and each edge $e(2)\in E_N(2)$ length $4/7$. Let $B^{{\tilde {\mathcal G}}'_N}$ be a standard Brownian motion on ${\tilde {\mathcal G}}'_N$, let $G_{{\tilde {\mathcal G}}'_N}(u,v)$ be the Green's function of $B^{{\tilde {\mathcal G}}'_N}$, and let $\{{\tilde {\eta}}'_{N, v}:v\in {\tilde {\mathcal G}}'_N\}$ be a continuous realization of the Gaussian free field on ${\tilde {\mathcal G}}'_N$ with covariances given by $G_{{\tilde {\mathcal G}}'_N}(u,v)$. The restriction of $\{{\tilde {\eta}}'_{N, v}:v\in {\tilde {\mathcal G}}'_N\}$ to $V_N$ is the same as the Gaussian free field $\{ {\eta}_{N, v}:v\in V_N\}$. Moreover, $\{{\tilde {\eta}}'_{N, v}:v\in {\tilde {\mathcal G}}'_N\}$ can be obtained from $\{ {\eta}_{N, v}:v\in V_N\}$ by, for each edge $e' = (x, y)\in E_N(1)\cup E_N(2)$, independently sampling a variance 2 Brownian bridge of the same length as $e'$ with values $\eta_{N, x}$ and $\eta_{N, y}$ at the endpoints.

We now describe a coupling between the random walk loop soup cluster $\mathcal L_{1/2, N}$ and a graph $\mathcal O$ obtained from $\{{\tilde {\eta}}'_{N, v}:v\in {\tilde {\mathcal G}}'_N\}$. Fix some $\lambda>0$. We say an edge $e\in E_N$ is \emph{open} if ${\tilde {\eta}}'_{N, v}>\lambda$ for all $v\in e(1)$ (which we denote as ${\tilde{\eta}}'_{N, e(1)}>\lambda$ for notation convenience) or  ${\tilde {\eta}}'_{N, v}<-\lambda$ for all $v\in e(1)$. Let $\mathcal O$ be the graph (seen as a subgraph of $\mathcal G_N = (V_N, E_N)$) induced by these open edges.
\begin{lemma}\label{coupling}
	For any $\lambda\geq 2$, the graph $\mathcal O$ is stochastically dominated by  $\mathcal L_{1/2, N}$. That is to say, we have $(1_{{\tilde{\eta}}'_{N, e(1)}>\lambda \text { or } {\tilde{\eta}}'_{N, e(1)}<-\lambda})_{e\in E_N}$ is stochastically dominated by $(1_{n_e>0})_{e\in E_N}$.
\end{lemma}
\begin{proof}
	Since $\{\frac 1 2 (\tilde {\eta}'_{N, v})^2\}_{v\in V_N}$ and $\{\hat {\mathcal L}^v_{1/2}\}_{v\in V_N}$ both have the same law as $\{\frac{1}{2} \eta_{N, v}^2\}_{v\in V_N}$, we only need to show the stochastic dominance of $\mathcal O$ when conditioned on the former, by $\mathcal L_{1/2, N}$ when conditioned on the latter (with the same realization).
	
	On one hand, conditioned on $\{\tilde {\eta}'_{N,v}\}_{v\in V_N}$ we see that  $1_{\tilde {\eta}'_{N, e(1)}>\lambda \text { or } \tilde {\eta}'_{N, e(1)}<-\lambda}$'s are independent Bernoulli variables with mean $p_e$'s, where (below $(B_t)$ is a Brownian motion with variance 2 at time 1)
	\begin{align*}
		p_e=
		\begin{cases}
			\P(B_t>\lambda, \forall t\in[0,4]\mid B_0=\tilde {\eta}'_{N, x}, B_4=\tilde {\eta}'_{N, y}), &\text{ if } \tilde {\eta}'_{N, x},\tilde {\eta}'_{N, y}>\lambda\,,\\
			\P(B_t<-\lambda, \forall t\in[0,4]\mid B_0=\tilde {\eta}'_{N, x}, B_4=\tilde {\eta}'_{N, y}),&\text{ if } \tilde {\eta}'_{N, x},\tilde {\eta}'_{N, y}<-\lambda\,,\\
			0,&\text{ otherwise.}
		\end{cases}
	\end{align*}
	In the case when $\tilde {\eta}'_{N, x},\tilde {\eta}'_{N, y}>\lambda$ (the other case is essentially the same), by the reflection principle we get that
	$$1-p_e=\mathrm{e}^{-\frac {(\tilde {\eta}'_{N, x}+\tilde {\eta}'_{N, y}-2\lambda)^2}{16}}/\mathrm{e}^{-\frac {(\tilde {\eta}'_{N, x}-\tilde {\eta}'_{N, y})^2}{16}}=\mathrm{e}^{-\frac 1 4(\tilde {\eta}'_{N, x}-\lambda)(\tilde {\eta}'_{N, y}-\lambda)}\,.$$
	Therefore, conditioned on $\{\frac 1 2 (\tilde {\eta}'_{N, v})^2=\ell_v\}_{v\in V_N}$, we have $(1_{\tilde {\eta}'_{N, e(1)}>\lambda \text { or } \tilde {\eta}'_{N, e(1)}<-\lambda})_{e\in E_N}$ is stochastically dominated by independent Bernoulli's with mean $p_e'$'s, where
	\begin{align*}
		p_e'=
		\begin{cases}
			1-\mathrm{e}^{-\frac 1 4(\sqrt{2\ell_x}-\lambda)(\sqrt{2 \ell_y}-\lambda)}, & \text{ if } \ell_x, \ell_y> \frac {\lambda^2}2\,,\\
			0, & \text{ otherwise.}
		\end{cases}
	\end{align*}
	
	On the other hand, conditioned on $\{\hat {\mathcal L}^v_{1/2}=\ell_v\}_{v\in V_N}$, the graph of $\mathcal L_{1/2, N}$ has the same law as $(1_{n_e>0})_{e\in E_N}$, where $(n_e)_{e\in E_N}$ follows the random current model (as in \eqref{eq-def-random-current}) with parameters $\beta_e= 2\sqrt{\ell_x\ell_y}$ on edge $e=(x,y)$. Note that if we further condition on the parities of $(n_e)_{e\in E_N}$, then $n(e)$'s are independent with distribution $F_{1,\beta_e}$ if $n(e)$ is odd and distribution $F_{2,\beta_e}$ if $n(e)$ is even. Here $F_{1,\beta_e}$ and $F_{2,\beta_e}$ are both probability distributions on nonnegative integers such that 
	$$F_{1,\beta_e}(n)=\frac {(\beta_e)^n}{n!\sinh {\beta_e}}\mbox{ for }n=1,3,5,\ldots \mbox{ and } F_{2,\beta_e}(n)=\frac {(\beta_e)^n}{n!\cosh {\beta_e}} \mbox{ for } n=0,2,4,\ldots$$
	This implies that, conditioned on $\{\hat {\mathcal L}^v_{1/2}=\ell_v\}_{v\in V_N}$ and the parities of $(n_e)_{e\in E_N}$, $1_{n_e>0}$'s are independent Bernoulli variables with mean $p_e''$'s, where $p_e''=1$ if $n(e)$ is odd and $p_e''=1-1/\cosh \beta_e$ if $n(e)$ is even. Furthermore, for $\ell_x,\ell_y>\frac {\lambda^2}2$ and $\lambda \geq 2$, we have
	\begin{align*}
		1/\cosh \beta_e\leq 2/\mathrm{e}^{\beta_e}&=2/\mathrm{e}^{2\sqrt{\ell_x \ell_y}}\\
		&\leq \mathrm{e}^{- \sqrt{\ell_x \ell_y}}\\
		&\leq \mathrm{e}^{-\frac 1 4(\sqrt{2\ell_x}-\lambda)(\sqrt{2\ell_y}-\lambda)}=1-p_e'\,.
	\end{align*}
	Therefore, conditioned on $\{\hat {\mathcal L}^v_{1/2}=\ell_v\}_{v\in V_N}$, we have $(1_{n_e>0})_{e\in E_N}$ stochastically dominates independent Bernoulli's with mean $p_e'$'s.
	
	Combining the above two parts completes the proof of the lemma.\qed
\end{proof}

\begin{remark}
It is worth pointing out that for our proof strategy to go through, it suffices as long as the law of the edge visits conditioned on vertex local times dominates the random current model as in \eqref{eq-def-random-current} for $\beta_e \geq c \sqrt{\ell_x \ell_y}$ for some fixed positive constant $c$ (since we can tune the resistance on $e(1)$ and $e(2)$). The fact that $c=2$ is of no importance to us. This flexibility may be useful when attempting to extend our proof strategy to some other contexts. 
\end{remark}

In light of Lemma ~\ref{coupling}, define (for $>\lambda$, the definition for $<-\lambda$ is similar)
$$ D_{\tilde {\eta}'_N,E_N(1), >\lambda} (u, v) = \min_{\gamma} |\gamma|\,,$$
where the minimum is over all path $\gamma\subseteq V_N\cup E_N(1)\subseteq \tilde {\mathcal G}'_N$ joining $u$ and $v$ such that $\tilde {\eta}'_{N, x}>\lambda$ for all $x\in\gamma$. In order to prove Theorem \ref{thm-random-walk-loop-soup-critical}, it suffices to prove the following proposition.
\begin{prop}\label{prop-key-statement}
	For any $0<\alpha<\beta < 1$, there exists a constant $c>0$ such that for all $N$
	$$\P(\min \{D_{\tilde {\eta}'_N,E_N(1), >\lambda}(\partial V_{\alpha N}, \partial V_{\beta N}), D_{\tilde {\eta}'_N,E_N(1), <-\lambda}(\partial V_{\alpha N}, \partial V_{\beta N})\}\leq N \mathrm{e}^{(\log N)^{\chi}}) \geq c\,.$$
\end{prop}
For the rigor of proof (when applying e.g., FKG inequality later), we will consider the following discrete approximation of the metric graph. We let $\Pi_N=(V_N\cup E_N(1))\cap \frac 1 {N^3} \mathbb Z^2$ be an $N^3$-discretization of $V_N\cup E_N(1)$, and write $\mathcal H_{>\lambda} = \{v\in \Pi_N: \tilde \eta'_{N, v} > \lambda + 1/N\}$ and $\mathcal H_{<-\lambda} = \{v\in \Pi_N: \tilde \eta'_{N, v} <- \lambda - 1/N\}$.
We
define  for $u, v\in \Pi_N$
$$ D_{\tilde {\eta}'_N,E_N(1), >\lambda, \Pi_N} (u, v) = \min_{\gamma} |\gamma|\,,$$
where the minimum is over all path $\gamma\subseteq V_N\cup E_N(1)\subseteq \tilde {\mathcal G}'_N$ joining $u$ and $v$ such that $\gamma \cap \Pi_N \subseteq \mathcal H_{>\lambda}$. We define $ D_{\tilde {\eta}'_N,E_N(1), <-\lambda, \Pi_N}$ similarly.  Analogous to \eqref{eq-mathcal-A-B-I}, we define for $i\geq 1$ (since the notation in this section is already complicated, we drop the tilde for notation convenience)
\begin{equation}\label{eq-A-B-I-loop-soup}
\begin{split}
 {\mathcal I}_i^{>\lambda} & = \{v\in  \Pi_N \setminus \tilde {\mathcal G}_{\alpha N}: D_{\tilde {\eta}'_N,E_N(1), >\lambda, \Pi_N}(v,  \partial V_{\alpha N} \cap \mathcal H_{>\lambda}) \leq i\}\,,\\
{\mathcal A}_i^{>\lambda} &= \{v\in V_N \setminus V_{\alpha N} : D_{\tilde {\eta}'_N,E_N(1), >\lambda, \Pi_N} (v, \partial V_{\alpha N} \cap \mathcal H_{>\lambda}) = i\}\,,\\
{\mathcal B}_i^{=\lambda} &= \{v\in  \Pi_N\setminus (\mathcal I_i^{>\lambda} \cup \tilde {\mathcal G}_{\alpha N}) : v \mbox{ has } \ell_1\mbox{-distance precisely } 1/N^3 \mbox{ from some } u\in \mathcal I_i^{>\lambda} \setminus \mathcal A_i^{>\lambda} \}\,,\\
\mathcal B_0&= \{v\in\partial V_{\alpha N}: |\tilde {\eta}'_{N, v}|\leq \lambda+\frac 1 N\}\,.
\end{split}
\end{equation}
 We note that in definitions of \eqref{eq-A-B-I-loop-soup}, we excluded the interior of $\tilde {\mathcal G}_{\alpha N}$ for convenience of later analysis. We also consider an analogous definition for the version of $ {\mathcal I}_i^{<-\lambda},
{\mathcal A}_i^{<-\lambda},
{\mathcal B}_i^{=-\lambda}$.  One can then verify that the sets of $\mathcal A_i^{>\lambda}, \mathcal B_i^{=\lambda}, \mathcal I_i^{>\lambda}$  as in \eqref{eq-A-B-I-loop-soup} (similar for the version of $<-\lambda$) can be constructed by the following \emph{exploration procedure}.

\noindent{\bf Exploration procedure.}  Initially we set
	\begin{align}
		\mathcal I^{>\lambda}_0 = \mathcal A^{>\lambda}_0 =\partial V_{\alpha N} \cap \mathcal H_{>\lambda},  \quad  \mathcal I^{<-\lambda}_0  = \mathcal A^{<-\lambda}_0 = \partial V_{\alpha N} \cap \mathcal H_{<-\lambda}\,,  \quad {\mathcal B}_0^{=\lambda} = {\mathcal B}_0^{=-\lambda} = \emptyset\,.  \label{eq-B-0-initial}
	\end{align}
	For $i=0,1,2,\ldots$, we run the exploration procedure inductively as follows:
	\begin{itemize}
	\item If $\mathcal A^{>\lambda}_i=\emptyset$, stop the exploration procedure and set  $\mathcal A^{>\lambda}_{j+1}=\emptyset$, $\mathcal B^{=\lambda}_{j+1}=\mathcal B^{=\lambda}_{i}$ and $\mathcal I^{>\lambda}_{j+1}=\mathcal I^{>\lambda}_i$  for all $j\geq i$.

		\item If $\mathcal A^{>\lambda}_i \neq \emptyset$, then for each $v\in \mathcal A^{>\lambda}_i$ and every edge $e(1)=(v,u)\in E_N(1)$ incident to $v$, if $u\in V_N\setminus V_{\alpha N}$ and  $e(1)\cap \Pi_N$ is not contained in  $\mathcal I^{>\lambda}_i$, we explore from $v$ along $e(1)$ to $u$ (note that we only explore vertices in $\Pi_N$) until we reach a point $w\in \Pi_N$ with $\tilde {\eta}'_{N, w}\leq \lambda+\frac 1 N$. If no such $w$ is encountered, then we will explore all $e(1)\cap \Pi_N$ including $u$.
		\item Let $\mathcal I_{i+1}^{>\lambda}$ be the union of $\mathcal I_i^{>\lambda}$ and all vertices explored in previous step with GFF values $> \lambda + 1/N$; let $\mathcal B_{i+1}^{=\lambda}$ be the union of $\mathcal B_i^{=\lambda}$ and all explored vertices in previous step with GFF value $\leq \lambda + 1/N$ (i.e., all the corresponding $w$'s); let $\mathcal A_{i+1}^{>\lambda}$ be the union of all lattice points which are explored in the previous step with GFF value $> \lambda + 1/N$ and furthermore are not in $\cup_{j=0}^{i}\mathcal A^{>\lambda}_j$.
	\end{itemize}
	 We employ a similar procedure for the version of $<-\lambda$.  
	
\begin{proof}[Proof of  Proposition~\ref{prop-key-statement}]
We will first describe an overview of the proof of the proposition, which consists of three steps. In Step 1, we decompose the event that the distance between the two boundaries is large into a disjoint union of events (see \eqref{disjoint-union-1}), where each  event corresponds to some constraints of the GFF values in a certain region (see \eqref{I_k}).  In Step 2, we will show that under the conditioning of the events obtained in Step 1, the conditional law of our ``observable'' (see \eqref{eq-observable}) can be controlled: its mean is shifted by at most a constant (see \eqref{eq-understand-conditioning}) while its variance is decreased by at least some positive constant (see \eqref{eq-conditional-variance-bound}). We remark that in this step Makarov's theorem (see \eqref{eq-harmonic-measure-J-4}) and entropic repulsion type of arguments (see Lemma~\ref{lem-control-X-3}) are employed.  In Step 3, we show that the union of the events obtained in Step 1 cannot occur with probability 1, since (as we will show) the law of a Gaussian variable (our observable) cannot be a mixture of Gaussians with bounded shifts on the mean and non-trivial reduction on the variance. 

\noindent {\bf Step 1. Decomposition of the event.} Let $\Lambda_{\mathrm{bad}} = \{\sup_{v\in V_N} |\tilde {\eta}'_{N, v}| \geq 100 \log N\}$ be as before. Define
	\begin{equation}\label{def-Lambda_c}
		\Lambda_c=\{|\tilde {\eta}'_{N, u}-\tilde {\eta}'_{N, v}|\leq \frac 1 N, \forall e(1)\in E_N(1) \text { and } u,v\in e(1) \text { such that } |u-v|\leq \frac 1 {N^3}\}\,.
	\end{equation}
	Since conditioned on $\{ \tilde {\eta}'_{N, v}:v\in V_N\}$, we have $\{ \tilde {\eta}'_{N, e(1)}:e(1)= (x, y)\in E_N(1)\}$ are independent variance 2 Brownian bridges (which are H\"older continuous of any order less than 1/2) of length 4 with values $\tilde {\eta}'_{N, x}$ and $\tilde {\eta}'_{N, y}$ at the endpoints, we see that
	\begin{equation}\label{eq-Lambda_c-large}
		\P(\Lambda_c)\geq\P(\Lambda_{\mathrm{bad}}^c)(1-o_N(1))=1-o_N(1)\,.
	\end{equation}
	Denote
	$$\mathcal E\deq\{\min \{D_{\tilde {\eta}'_N,E_N(1), >\lambda}(\partial V_{\alpha N}, \partial V_{\beta N}), D_{\tilde {\eta}'_N,E_N(1), <-\lambda}(\partial V_{\alpha N}, \partial V_{\beta N})\}> N \mathrm{e}^{(\log N)^{\chi}}\}\,.$$
	Suppose that both events $\mathcal E$ and $\Lambda_c$ occur, then by \eqref{def-Lambda_c} we must have that $\mathcal I^{>\lambda}_i$ and $\mathcal I^{<-\lambda}_i$ are both disjoint from $V_N\setminus V_{\beta N}$ for all $0\leq i< N \mathrm{e}^{(\log N)^{\chi}}$. Further, since all of the $\mathcal A^{>\lambda}_i$ and $\mathcal A^{<-\lambda}_i$ where $0\leq i< N \mathrm{e}^{(\log N)^{\chi}}$ are disjoint from each other, we see that there exists at least an $i_0 < N \mathrm{e}^{(\log N)^{\chi}}$ such that 
	\begin{equation}
		|\mathcal A^{>\lambda}_{i_0}\cup \mathcal A^{<-\lambda}_{i_0}| \leq N \mathrm{e}^{-(\log N)^{\chi}}\,.
	\end{equation}
	Moreover (still on the event $\mathcal E\cap\Lambda_c$), for all $1\leq i< N \mathrm{e}^{(\log N)^{\chi}}$, we have by \eqref{def-Lambda_c} again
	$$\lambda<\tilde {\eta}'_{N, v}\leq \lambda+\frac 1 N \ \mbox{for all } v\in \mathcal B^{=\lambda}_{i} \quad \text { and } \quad -\lambda-\frac 1 N\leq\tilde {\eta}'_{N, v}<-\lambda \ \mbox{ for all } v\in \mathcal B^{=-\lambda}_{i}\,.$$
	
	We claim that for any $k\geq0$, from $\mathcal I^{>\lambda}_k = I^{>\lambda}_k$ and $\mathcal I^{<-\lambda}_k = I^{<-\lambda}_k$ we can determine (uniquely) the sets $\mathcal A^{>\lambda}_i , \mathcal A^{<-\lambda}_i , \mathcal B^{=\lambda}_{i}, \mathcal B^{=-\lambda}_{i}$ for all $1\leq i\leq k$ as well as $\mathcal B_0$. This is because $D_{\tilde {\eta}'_N,E_N(1), >\lambda, \Pi_N}(v,  \partial V_{\alpha N} \cap \mathcal H_{>\lambda})$  for $v\in \mathcal I^{>\lambda}_k $ is the same as the graph distance on $\mathcal I^{>\lambda}_k$ between $v$ and  $\partial V_{\alpha N} \cap \mathcal H_{>\lambda}$, and thus $\mathcal I^{>\lambda}_i$ for $0\leq i\leq k$ can be determined (similar for the version of $<-\lambda$) and therefore all the other sets can be determined (see \eqref{eq-A-B-I-loop-soup}). We denote $\mathcal A^{>\lambda}_i , \mathcal A^{<-\lambda}_i , \mathcal B^{=\lambda}_{i}, \mathcal B^{=-\lambda}_{i}$ for all $1\leq i\leq k$ and  $\mathcal B_0$ as $A^{>\lambda}_i, A^{<-\lambda}_i, B^{=\lambda}_{i}, B^{=-\lambda}_{i}$ for all $1\leq i\leq k$ and $B_0$, respectively --- they are all functions of $I^{>\lambda}_k$ and $I^{<-\lambda}_k$. We let $\mathcal P_k$ denote all $(I^{>\lambda}_k, I^{<-\lambda}_k)$ such that $I^{>\lambda}_k$ and $I^{<-\lambda}_k$ are both disjoint from $V_N\setminus V_{\beta N}$ and are feasible realizations of $\mathcal I^{>\lambda}_k$ and $\mathcal I^{<-\lambda}_k$, and such that
	\begin{equation}\label{eq-min-i0}
		 \min\{i_0: |A^{>\lambda}_{i_0}\cup   A^{<-\lambda}_{i_0}| \leq N \mathrm{e}^{-(\log N)^{\chi}}\}=k\,.
	\end{equation}
	Denote $\temp\deq B_0\cup  B^{=\lambda}_k\cup B^{=-\lambda}_k\cup I^{>\lambda}_k \cup I^{<-\lambda}_k$.  By our exploration procedure perspective of \eqref{eq-A-B-I-loop-soup}, we have
	$$\{\mathcal I^{>\lambda}_k = I^{>\lambda}_k , \mathcal I^{<-\lambda}_k = I^{<-\lambda}_k \} \in \mathcal F_{\temp} = \sigma(\{\tilde \eta'_{N, v}: v\in \temp\})\,.$$
	In summary of the discussions above, we have
	\begin{equation}\label{disjoint-union-1}
		\mathcal E\cap\Lambda_c\subseteq \bigcup_{\substack{0\leq k< N \mathrm{e}^{(\log N)^{\chi}}\\ (I^{>\lambda}_k, I^{<-\lambda}_k)\in \mathcal P_k}} \mathcal E_{I^{>\lambda}_k, I^{<-\lambda}_k}
	\end{equation}
	where
	\begin{align}\label{I_k}
		\mathcal E_{I^{>\lambda}_k, I^{<-\lambda}_k}=\{ &\tilde {\eta}'_{N, v}>\lambda+\frac 1 N \text{ for all }v\in  I^{>\lambda}_k, \quad \tilde {\eta}'_{N, v}<-\lambda-\frac 1 N \text{ for all } v\in  I^{<-\lambda}_k,\nonumber\\
		& \lambda<\tilde {\eta}'_{N, v}\leq \lambda+\frac 1 N \text{ for all } v\in B^{=\lambda}_k, \, -\lambda-\frac 1 N\leq\tilde {\eta}'_{N, v}<-\lambda \text{ for all } v\in  B^{=-\lambda}_k,\nonumber\\
		& |\tilde {\eta}'_{N, v}|\leq\lambda+\frac 1 N \text{ for all } v\in  B_0\}\,.
	\end{align}
	
	\medskip

\noindent {\bf Step 2: Influence on our observable from the conditioning.} Note that the event  $\mathcal E_{I^{>\lambda}_k, I^{<-\lambda}_k}$ is $\mathcal F_{\temp}$-measurable. 
	Conditioned on $\mathcal F_{\temp}$, the field $\{\tilde {\eta}'_{N, u}: u\in \tilde {\mathcal G}'_N \setminus \temp\}$ is distributed as a GFF with boundary condition $\{\tilde {\eta}'_{N, v}: v\in \temp\}$ and zero on $\partial V_N$. In particular, for each $u\in \partial V_{(1+\beta) N/2}$, we have
	\begin{equation}\label{eq-expression-conditional-expectation-tilde-tilde}
		\E(\tilde {\eta}'_{N, u} \mid \mathcal F_{\temp}) = \sum_{v\in \temp} \widetilde{\mathrm{Hm}}'(u, v; \temp \cup \partial V_N) \cdot \tilde {\eta}'_{N, v}\,.
	\end{equation}
	Here $\widetilde{\mathrm{Hm}}'(u, v; K)$ denotes the harmonic measure of $B^{\tilde {\mathcal G}'_N}$ at $v$ with respect to starting point $u$ and target set $K$.
	
From our exploration procedure, we know that for any $u\in \partial V_{(1+\beta) N/2}$, we have $\{v\in \temp:\widetilde{\mathrm{Hm}}'(u, v; \temp \cup \partial V_N)\neq 0\}\subseteq J_1\cup J_2\cup J_3\cup J_4$, which can be described as follows.
	\begin{itemize}
		\item $J_1=B_0$.
		\item $J_2=  B^{=\lambda}_k \cup  B^{=-\lambda}_k$.
		\item For each $v\in J_3$, we have $v\in A^{>\lambda}_i$ (or $v\in A^{<-\lambda}_i$ respectively) for some $0\leq i\leq k-1$, and on an edge $e(1)=(v,v')\in E_N(1)$ there is a $w\in B^{=\lambda}_{i+1} \subseteq B^{=\lambda}_k$ (or $w\in B^{=-\lambda}_{i+1}$ respectively). In particular, each $v\in J_3$ must satisfy that $v\in V_N$ and that $v$ has Euclidean distance less than 1 to a point $w\in J_2$.
		\item $J_4=A^{>\lambda}_k\cup A^{<-\lambda}_k$.
	\end{itemize}
	Define
	\begin{equation}\label{eq-observable}
	X=\frac {1}{ |\partial V_{(1+\beta) N/2}|}\sum_{u\in  \partial V_{(1+\beta) N/2}} \tilde {\eta}'_{N, u}\,.
	\end{equation}
	Note that the preceding definition of $X$ is consistent with that of  \eqref{eq-define-X}, since $\{\tilde \eta'_{N, u}: u\in V_N\}$ has the same law as $\{\eta_{N, u}: u\in V_N\}$.
	Further define
	$$X_1=\frac {1}{ |\partial V_{(1+\beta) N/2}|}\sum_{u\in  \partial V_{(1+\beta) N/2}}\sum_{v\in J_1} \widetilde{\mathrm{Hm}}'(u, v; \temp \cup \partial V_N) \cdot \tilde {\eta}'_{N, v}\,,$$
	$$X_2=\frac {1}{ |\partial V_{(1+\beta) N/2}|}\sum_{u\in  \partial V_{(1+\beta) N/2}}\sum_{v\in J_2} \widetilde{\mathrm{Hm}}'(u, v; \temp \cup \partial V_N) \cdot \tilde {\eta}'_{N, v}\,,$$
	$$X_3=\frac {1}{ |\partial V_{(1+\beta) N/2}|}\sum_{u\in  \partial V_{(1+\beta) N/2}}\sum_{v\in J_3} \widetilde{\mathrm{Hm}}'(u, v; \temp \cup \partial V_N) \cdot \tilde {\eta}'_{N, v}\,,$$
	$$X_4=\frac {1}{ |\partial V_{(1+\beta) N/2}|}\sum_{u\in  \partial V_{(1+\beta) N/2}}\sum_{v\in J_4} \widetilde{\mathrm{Hm}}'(u, v; \temp \cup \partial V_N) \cdot \tilde {\eta}'_{N, v}\,.$$
	Then by \eqref{eq-expression-conditional-expectation-tilde-tilde}, we have
	$$\E(X \mid \mathcal F_{\temp})=X_1+X_2+X_3+X_4\,.$$
	It is clear that $|X_1|\leq \lambda+\frac 1 N$ and $|X_2|\leq \lambda+\frac 1 N$ always hold. Let $D_k=(\cup_{i=0}^{k}A^{>\lambda}_i)\cup(\cup_{i=0}^{k}A^{<-\lambda}_i)\cup B_0$ be the set of all the lattice points in $\temp$ (so $D_k \supseteq \partial V_{\alpha N}$ and $D_k$ is connected). Then $J_4\subseteq D_k\subseteq \temp$, so that 
	\begin{equation}
		\widetilde{\mathrm{Hm}}'(u, J_4; \temp \cup \partial V_N) \leq \widetilde{\mathrm{Hm}}'(u, J_4 ; D_k \cup \partial V_N)\,.
	\end{equation}
	Since $|J_4|\leq N \mathrm{e}^{-(\log N)^{\chi}}$ by \eqref{eq-min-i0}, we deduce from Lemma~\ref{lem-Makarov} that (recall $D_k \supseteq \partial V_{\alpha N}$ and $D_k$ is connected)
	\begin{equation}\label{eq-harmonic-measure-J-4}
	\widetilde{\mathrm{Hm}}'(u, J_4 ; D_k \cup \partial V_N)=\mathrm{Hm}(u,  J_4;  D_k \cup \partial V_N) \leq \mathrm{Hm}(u,  J_4;  D_k)= o(\log N)^{-10}\,,
	\end{equation}
	and therefore $\widetilde{\mathrm{Hm}}'(u, J_4; \temp \cup \partial V_N)= o(\log N)^{-10}$. Recall that $\Lambda_{\mathrm{bad}} = \{\sup_{v\in V_N} |\tilde {\eta}'_{N, v}| \geq 100 \log N\}$. Then if the event $\Lambda_{\mathrm{bad}}$ does not occur, we have
	$|X_4|=o(\log N)^{-8}$.
	
	It now remains to control $X_3$ on the event $\mathcal E_{I^{>\lambda}_k, I^{<-\lambda}_k}$. We will show in Lemma~\ref{lem-control-X-3} below that there exists a $\mathcal F_{\temp}$-measurable event $\mathcal E'_{I^{>\lambda}_k, I^{<-\lambda}_k}\subseteq \mathcal E_{I^{>\lambda}_k, I^{<-\lambda}_k}$ such that
	\begin{equation}\label{eq-understand-conditioning}
	  \E(X \mid \mathcal F_{\temp}) \leq \Delta \mbox{ on the event } \mathcal E'_{I^{>\lambda}_k, I^{<-\lambda}_k}\setminus\Lambda_{\mathrm{bad}}\,,
	 \end{equation}
where $\Delta>0$ is a constant depending only on $\lambda$  and moreover
	\begin{equation}\label{eq-high-probability}
		\P(\mathcal E'_{I^{>\lambda}_k, I^{<-\lambda}_k})\geq (1-o_N(1))\P(\mathcal E_{I^{>\lambda}_k, I^{<-\lambda}_k})\,.
	\end{equation}
	In addition, recall that $\var X\leq c_4$ by \eqref{Var-X} and note that on the event $\mathcal E'_{I^{>\lambda}_k, I^{<-\lambda}_k}\setminus\Lambda_{\mathrm{bad}}$, by \eqref{random-walk-lemma-3-eq} we have 
	\begin{equation}\label{eq-conditional-variance-bound}
	\var X-\var( X \mid \mathcal F_{\temp})\geq c_{11}>0\,.
	\end{equation}  
	
	\medskip
	
\noindent {\bf Step 3: Gaussian v.s. a mixture of Gaussians.} Let $t=\Delta+s {\sqrt{\var X-c_{11}}}$ for $s>0$. Then on the event $\mathcal E'_{I^{>\lambda}_k, I^{<-\lambda}_k}\setminus\Lambda_{\mathrm{bad}}$, we have
	$$\P(X\leq t \mid \mathcal F_{\temp}) \geq \P(Z(\Delta,\var X-c_{11})\leq t)=\P(Z\leq s)\,,$$
	where $Z(\Delta, \var X - c_{11})$ is a Gaussian variable with mean $\Delta$ and variance $\var X - c_{11}$, and $Z$ is a standard Gaussian variable. Therefore (since $\mathcal E'_{I^{>\lambda}_k, I^{<-\lambda}_k}$ is $\mathcal F_{\temp}$-measurable)
	$$\P(X \leq t, \mathcal E'_{I^{>\lambda}_k, I^{<-\lambda}_k})=\E (\P(X\leq t \mid \mathcal F_{\temp})\1_{\mathcal E'_{I^{>\lambda}_k, I^{<-\lambda}_k}})\geq\P(Z\leq s)\P(\mathcal E'_{I^{>\lambda}_k, I^{<-\lambda}_k}\setminus\Lambda_{\mathrm{bad}})\,.$$
	Summing this over all $0\leq k< N \mathrm{e}^{(\log N)^{\chi}}$ and all $(I^{>\lambda}_k, I^{<-\lambda}_k)\in \mathcal P_k$, we have
	$$\P(X \leq t)\geq \P(Z\leq s) \P((\bigcup_{\substack{0\leq k< N \mathrm{e}^{(\log N)^{\chi}}\\ (I^{>\lambda}_k, I^{<-\lambda}_k)\in \mathcal P_k}} \mathcal E'_{I^{>\lambda}_k, I^{<-\lambda}_k})\setminus \Lambda_{\mathrm{bad}})\,.$$
	Therefore, for a sufficiently large constant $s>0$ and a constant $c'>0$, we have
	$$\P(\bigcup_{\substack{0\leq k< N \mathrm{e}^{(\log N)^{\chi}}\\ (I^{>\lambda}_k, I^{<-\lambda}_k)\in \mathcal P_k}} \mathcal E'_{I^{>\lambda}_k, I^{<-\lambda}_k})\leq \frac {\P(X \leq t)}{\P(Z\leq s)}+\P(\Lambda_{\mathrm{bad}})\leq 1-c'\,.$$
	Combined with \eqref{eq-high-probability}, \eqref{disjoint-union-1} and \eqref{eq-Lambda_c-large}, this gives us the result of the proposition.\qed
\end{proof}

\begin{lemma}\label{lem-control-X-3}
There exists a $\mathcal F_{\temp}$-measurable event $\mathcal E'_{I^{>\lambda}_k, I^{<-\lambda}_k}\subseteq \mathcal E_{I^{>\lambda}_k, I^{<-\lambda}_k}$ such that \eqref{eq-understand-conditioning} and \eqref{eq-high-probability} hold.
\end{lemma}
\begin{proof}
The proof of Lemma~\ref{lem-control-X-3} constitutes the rest of the paper. To this end, we consider any fixed numbers $\{x_{N,v}\}_{v\in B_0\cup  B^{=\lambda}_k \cup  B^{=-\lambda}_k}$ such that 
	\begin{align*}
		|x_{N, v}|&\leq\lambda+\frac 1 N \ \mbox{ for all } v\in  B_0, \quad \lambda<x_{N, v}\leq \lambda+\frac 1 N \mbox{ for all } v\in  B^{=\lambda}_k,\\
		-\lambda-\frac 1 N& \leq x_{N, v}<-\lambda \mbox{ for all }  v\in  B^{=-\lambda}_k\,.
	\end{align*}
	We define three events $H_=, H_+, H_-$ as follows:
	\begin{itemize}
		\item $H_= = \{\tilde {\eta}'_{N, v}=x_{N,v} \text{ for all } v\in  B_0\cup  B^{=\lambda}_k \cup B^{=-\lambda}_k\}$;
		\item $H_+ = \{\tilde {\eta}'_{N, v}>\lambda+\frac 1 N \text{ for all }v\in  I^{>\lambda}_k\}$;
		\item $H_- = \{\tilde {\eta}'_{N, v}<-\lambda-\frac 1 N \text{ for all } v\in  I^{<-\lambda}_k\}$.
	\end{itemize}
	We will show in Lemmas~\ref{X3-conditional-expectation} and \ref{X3-conditional-variance} below that conditioned on $H_=\cap H_-\cap H_+$, we have
	\begin{enumerate}[(a)]
		\item $\E(X_3\mid H_=, H_+, H_-) \leq \lambda + \frac 1 N + C_0$ where $C_0$ is a constant; \label{eq-a}
		\item $\var(X_3\mid H_=, H_+, H_-)=o_N(1)$ (where we use $o_N(1)$ to denote a quantity that only depends on $N$ and tends to 0 as $N\to \infty$). \label{eq-b}
	\end{enumerate}
	As a corollary of \eqref{eq-a} and \eqref{eq-b}, conditioned on $\{H_=,H_+,H_-\}$, we have $X_3$ itself is bounded from above by $\lambda + C_0 + 1$ with probability $(1-o_N(1))$. By integrating over all $\{x_{N,v}\}_{v\in B_0\cup   B^{=\lambda}_k \cup  B^{=-\lambda}_k}$, we conclude the proof of the lemma. \qed
\end{proof}
It remains to prove \eqref{eq-a} and \eqref{eq-b}, which are incorporated in Lemmas \ref{X3-conditional-expectation} and \ref{X3-conditional-variance} below.
\begin{lemma}\label{X3-conditional-expectation}
	There exists a constant $C_0>0$ such that for any $v\in J_3$,  we have
	$$\E(\tilde {\eta}'_{N, v}\mid H_=, H_+, H_-)\leq \lambda+\frac 1 N +C_0\,.$$
\end{lemma}
Before proving Lemma~\ref{X3-conditional-expectation}, we first prove a technical lemma on the existence of a certain harmonic function, which follows from a modification of a standard result. For any $w\in \tilde {\mathcal G}'_N$, we say a function $f$ is harmonic on $\tilde {\mathcal G}'_N\setminus (\partial V_N \cup \{w\})$ if the restriction of $f$ on $V_N\cup\{w\}$ (i.e.,  $f(u)$ restricted on $u\in V_N\cup\{w\}$) is harmonic on the discrete graph $(V_N\cup\{w\}, E'_{N,w})$ except at $w$ and $\partial V_N$, and $f$ is linear on each segment $e'\in E'_{N,w}$, where $E'_{N,w}=\{e': e'\in E_N(1)\cup E_N(2) \text { and } w\notin e'\} \cup \{(v,w),(w, v')\}$. We remark that in our application $w$ is given in the definition of $J_3$ (in Step 2 of the proof for  Proposition~\ref{prop-key-statement}), and our intuition is that two nearby points cannot take GFF values that are too different from each other (even under the conditioning of reasonable events). This intuition is reflected partly in \eqref{f(u)-bound} below. 
\begin{lemma}\label{lem-harmonic-function}
 For any $w\in \tilde {\mathcal G}'_N$ and any constants  $C, C_1>0$ that can be taken to be arbitrarily large, there exists a function $f$ defined on $\tilde {\mathcal G}'_N$ such that $f$ is harmonic on $\tilde {\mathcal G}'_N\setminus (\{w\} \cup \partial V_N)$. In addition, 
	\begin{equation}\label{f(u)-bound}
		|f(u)-C\log (|u-w|+2)-C_1|\leq L(C) \quad \text{ for all } u\in \tilde {\mathcal G}'_N\,,
	\end{equation}
	where $L(C)$ depends only on $C$. In particular, we take $C_1>L(C)$ so that $f(u)>0$ for all $u\in \tilde {\mathcal G}'_N$.
\end{lemma}
\begin{proof}
By \cite[(B17)]{Dunlop1992} or \cite[Theorem 4.4.4]{LL10}, there exist a function $g$ defined on $\mathbb{Z}^2$ and absolute constants $C, C'_1>0$ that can be taken to be arbitrarily large, such that $g$ is harmonic on $\mathbb{Z}^2\setminus \{(0,0)\}$ and $|g(u)-C\log (|u|+2)-C'_1|\leq L'(C)$ for all $u\in \mathbb{Z}^2$ (where $L'(C)$ is a function that only depends on $C$) --- in fact, the function $g$ is a multiple of the potential kernel for the simple random walk on $\mathbb Z^2$. Now let us define for $u\in \tilde {\mathcal G}'_N$
	\begin{equation*}
		f(u)=
		\begin{cases}
			|v'-w| g(u-v)+|v-w|g(u-v'), & \text{ if } u\in V_N\,,\\
			f(w),& \text{ if } u=w\,,\\
			\text{linear interpolation between $f(x)$ and $f(y)$}, & \text{ if } u\in e'=(x,y)\in E'_{N,w}\,,
		\end{cases}
	\end{equation*}
	where
	$$f(w)=(|v-w|^2+|v'-w|^2)g((0,0))-8|v-w||v'-w|Dg((0,0))+|v-w||v'-w|(g(v'-v)+g(v-v'))$$
	and
	$$Dg((0,0))=g((0,1))+g((0,-1))+g((1,0))+g((-1,0))-4g((0,0))\,.$$
	Then by definition, $f(u)$ is clearly harmonic on $\tilde {\mathcal G}'_N\setminus (\{v,v',w\} \cup \partial V_N)$. To show that it is also harmonic at $v$ and $v'$, we have to verify that
	$$(3+\frac 7 8+\frac 1 {8|v-w|})f(v)=\sum_{i=1}^3 f(v_i)+\frac 7 8 f(v')+\frac 1 {8|v-w|} f(w)$$
	and
	$$(3+\frac 7 8+\frac 1 {8|v'-w|})f(v')=\sum_{i=1}^3 f(v'_i)+\frac 7 8 f(v)+\frac 1 {8|v'-w|} f(w)$$
	where $v_1,v_2,v_3\in V_N$ are the three neighbors of $v$ other than $v'$, and $v'_1,v'_2,v'_3\in V_N$ are the three neighbors of $v'$ other than $v$. We give the details for verification of the first identity (the second one is similar) as follows:
	\begin{eqnarray*}
		&&\sum_{i=1}^3 f(v_i)+\frac 7 8 f(v')+\frac 1 {8|v-w|} f(w)\\
		&=&\sum_{i=1}^3 (|v'-w| g(v_i-v)+|v-w|g(v_i-v'))+\frac 7 8 (|v'-w| g(v'-v)+|v-w|g((0,0)))+\frac 1 {8|v-w|} f(w)\\
		&=&|v'-w|(Dg((0,0))+4g((0,0))-g(v'-v))+|v-w|(4g(v-v')-g((0,0)))\\
		&&+\frac 7 8 (|v'-w| g(v'-v)+|v-w|g((0,0)))+\frac 1 {8|v-w|} f(w)\\
		&=&(3+\frac 7 8+\frac 1 {8|v-w|})(|v'-w|g((0,0))+|v-w|g(v-v'))\\
		&=&(3+\frac 7 8+\frac 1 {8|v-w|})f(v)\,,
	\end{eqnarray*}
	where the penultimate equality follows by comparing the coefficients of $g((0,0))$, $Dg((0,0))$, $g(v'-v)$ and $g(v-v')$. For completeness, we record the detailed computations on these coefficients here:
	\begin{eqnarray*}
		g((0,0)):&& 4|v'-w|+\frac 1 {8|v-w|}|v'-w|^2-(3+\frac 7 8+\frac 1 {8|v-w|})|v'-w|\\
		&=&|v'-w|(4+\frac {1-|v-w|} {8|v-w|}-(3+\frac 7 8+\frac 1 {8|v-w|})) = 0\\
		&&\mbox{ and }
		-|v-w|+\frac 7 8 |v-w|+\frac 1 {8|v-w|}|v-w|^2=0\,;\\
		Dg((0,0)): &&|v'-w|+\frac 1 {8|v-w|}(-8|v-w||v'-w|)=0\,;\\
		g(v'-v):&&-|v'-w|+\frac 7 8|v'-w|+\frac 1 {8|v-w|}|v-w||v'-w|=0\,;\\
		g(v-v'):&& 4|v-w|+\frac 1 {8|v-w|}|v-w||v'-w|-(3+\frac 7 8+\frac 1 {8|v-w|})|v-w|\\
		&=&|v-w|(4+\frac {1-|v-w|} {8|v-w|}-(3+\frac 7 8+\frac 1 {8|v-w|})) =  0\,.
	\end{eqnarray*}
	Therefore we completed the verification that $f(u)$ is harmonic on $\tilde {\mathcal G}'_N\setminus \{w\}$, and \eqref{f(u)-bound} follows easily from our definition of $f(u)$. \qed
	\end{proof}
\begin{proof}[Proof of Lemma~\ref{X3-conditional-expectation}]
	If for some $0\leq i\leq k-1$, $v\in A^{<-\lambda}_i$, then clearly we have $\E(\tilde {\eta}'_{N, v}\mid H_=, H_+, H_-)\leq -\lambda-\frac 1 N\leq \lambda+\frac 1 N +C_0$. So in what follows we assume that (recall the definition of $J_3$ in Step 2 of the proof for  Proposition~\ref{prop-key-statement}) for some $0\leq i\leq k-1$, $v\in A^{>\lambda}_i$ and  there exists a $w\in B^{=\lambda}_{i+1} \subseteq B^{=\lambda}_k$ on an edge $e(1)=(v,v')\in E_N(1)$ --- in later analysis $w$ will serve as a point where the value of the GFF is pinned (see \eqref{FKG}), and since $v$ is close to $w$ this intuitively implies that the GFF value at $v$ cannot be too large.  The type of argument in what follows is known as the entropic repulsion estimates in the presence of a hard wall \cite{Dunlop1992,BDG01}. Our context is close to \cite{Dunlop1992} with some slight complication, and our proof essentially follows from the same line of arguments.

By Lemma~\ref{lem-harmonic-function},  there exists a positive function $f$ (which we choose) defined on $\tilde {\mathcal G}'_N$ such that $f$ is harmonic on $\tilde {\mathcal G}'_N\setminus ( \{w\} \cup \partial V_N)$ and \eqref{f(u)-bound} holds.
	We now claim that
	\begin{multline}\label{FKG}
		\E(\tilde {\eta}'_{N, v}\mid H_=, H_+, H_-)\leq \E(\tilde {\eta}'_{N+2, v}\mid \tilde {\eta}'_{N+2, w}=f(w)+\lambda+\frac 1 N,\ \tilde {\eta}'_{N+2, u}=f(u)+\lambda+\frac 1 N \ \forall u\in \partial V_N,\\
		\tilde {\eta}'_{N+2, u}>\lambda+\frac 1 N \ \forall u\in \Pi_N\setminus \{w\})\,.
	\end{multline}
	We remark that on the right hand side of \eqref{FKG}, we considered $\{\tilde \eta'_{N+2, \cdot}\}$ for the reason that formally the process $\{\tilde \eta'_{N, \cdot}\}$ takes value 0 on $\partial V_N$ and thus we are not allowed to condition on non-zero values on $\partial V_N$ for $\{\tilde \eta'_{N, \cdot}\}$. We further note that the law of $\{\tilde \eta'_{N+2, v}: v\in \mathcal G'_N\}$ under the conditioning of $\tilde \eta'_{N+2, v} = 0$ for $v\in \partial V_N$ is the same as the law of $\{\tilde \eta'_{N, v}: v\in \mathcal G'_N\}$.
	In order to show \eqref{FKG}, we follow \cite[Appendix B.1]{giacomin2001aspects}. We let $\mu$ be the law of $\{\tilde {\eta}'_{N+2, u}:u\in \Pi_N\}$ and we see that $\mu$ has density $\mu(\,dr)=\exp(-H(r))\,dr$ (here $r=(r_u)_{u\in \Pi_N}$ denotes a general $|\Pi_N|$ dimensional vector) such that for every $r, r'\in \mathbb R^{|\Pi_N|}$
	$$H(r\vee r')+H(r\wedge r')\leq H(r)+H(r')\,,$$
	where $\vee$ and $\wedge$ are intended coordinate by coordinate. 
	For $q>0$, we define
	\begin{equation}\label{def-UVW}
		U^{(q)}(t)=
		\begin{cases}
			qt^4, & \text{if } t<0\\
			0, & \text{if } t\geq 0\\
		\end{cases}
		,\quad
		V^{(q)}(t)=
		\begin{cases}
			0, & \text{if } t<0\\
			qt^4, & \text{if } t\geq 0\\
		\end{cases}
		,\quad
		W^{(q)}(t)=qt^4
	\end{equation}
	and
	\begin{align*}
		\mu_1^{(q)}(\,dr)\propto\exp (&-\sum_{u\in B_0\cup  (\cup _{i=1}^k B^{=\lambda}_i) \cup (\cup _{i=1}^k B^{=-\lambda}_i)} W^{(q)}(r_u-x_{N,u})-\sum_{u\in  I^{>\lambda}_k}U^{(q)}(r_u-\lambda-\frac 1 N)\\
		&-\sum_{u\in I^{<-\lambda}_k}V^{(q)}(r_u+\lambda+\frac 1 N)-\sum_{u\in \partial V_N} W^{(q)}(r_u))\mu(\,dr)\,,\\
		\mu_2^{(q)}(\,dr)\propto\exp (&-\sum_{u\in \{w\}\cup\partial V_N} W^{(q)}(r_u-f(u)-\lambda-\frac 1 N)-\sum_{u\in  \Pi_N\setminus \{w\}}U^{(q)}(r_u-\lambda-\frac 1 N))\mu(\,dr)\,.
	\end{align*}
	It is not hard to verify that for any real numbers $t_0<t_1$ and any pair of functions
	\begin{align*}
		(h_1(t),h_2(t))\in \{&(W^{(q)}(t-t_0), W^{(q)}(t-t_1)), (W^{(q)}(t-t_0),U^{(q)}(t-t_0)), (V^{(q)}(t-t_0),U^{(q)}(t-t_0)), \\
		&(0, U^{(q)}(t-t_0)), (U^{(q)}(t-t_0), U^{(q)}(t-t_1))\}\,,
	\end{align*}
	we have for every $t, t'\in \mathbb R$,
	$$h_2(t\vee t')+h_1(t\wedge t')\leq h_2(t)+h_1(t')\,,$$
	and therefore for any $q>0$,
	it follows from \cite{Preston74} (see also  \cite[Appendix B.1]{giacomin2001aspects}) that $\mu_1^{(q)}$ is stochastically smaller than $\mu_2^{(q)}$ ($\mu_1^{(q)}\prec\mu_2^{(q)}$), i.e., for any increasing function $F$ one has $\mu_1^{(q)}(F) \leq \mu_2^{(q)}(F)$. As $q\to \infty$, $\mu_1^{(q)}$ and $\mu_2^{(q)}$ will converge weakly to the conditional laws on the left and right hand sides of \eqref{FKG}, respectively. Therefore \eqref{FKG} is verified.
	
	Clearly, the right hand side of \eqref{FKG} equals
	$$\lambda+\frac 1 N+\E_{\tilde {\eta}'_{N+2, w}=f(w),\ \tilde {\eta}'_{N+2, u}=f(u) \ \forall u\in \partial V_N}(\tilde {\eta}'_{N+2, v}\mid \tilde {\eta}'_{N+2, u}>0\text{ for all } u\in \Pi_N\setminus \{w\})\,.$$
	Denote by $M$ the boundary condition $\tilde {\eta}'_{N+2, w}=f(w),\ \tilde {\eta}'_{N+2, u}=f(u) \ \forall u\in \partial V_N$. Now under $M$, for any $u\in\tilde {\mathcal G}'_N$, we have $\tilde {\eta}'_{N+2, u}$ is Gaussian with mean $\E_M(\tilde {\eta}'_{N+2, u})=f(u)$ and variance $\var_M (\tilde {\eta}'_{N+2, u})=G_{\tilde {\mathcal G}'_N\setminus \{w\}}(u,u)$. It is well known that there exists a constant $C_2>0$ such that $G_{V_N\setminus \{v\}}(u,u)\leq C_2 \log (|u-v|+2)$ for all $u\in V_N$. Therefore we have for all $u\in V_N$,
	\begin{equation}\label{variance-bound}
		\var_M (\tilde {\eta}'_{N+2, u})=G_{\tilde {\mathcal G}'_N\setminus \{w\}}(u,u)\leq 32(1+G_{V_N\setminus \{v\}}(u,u))\leq 32(1+C_2 \log (|u-v|+2))\,.
	\end{equation}
	In particular, for $u=v$ we have the following bound (using \eqref{f(u)-bound} and \eqref{variance-bound})
	\begin{equation}\label{mean-on-positivity}
		\E_{M}(\tilde {\eta}'_{N+2, v} \1_{\tilde {\eta}'_{N+2, u}>0 \text{ for all } u\in \Pi_N\setminus \{w\}})\leq \E_M(|\tilde {\eta}'_{N+2, v}|)\leq C_3\,,
	\end{equation}
	where $C_3$ is a positive constant which only depends on $C$ and $C_1$.
	%
	%
	
	It now remains to bound $\P_M(\tilde {\eta}'_{N+2, u}>0 \text{ for all } u\in \Pi_N\setminus \{w\})$ from below. We will do this by giving a lower bound of $\P_M(\tilde {\eta}'_{N+2, u}>0 \text{ for all } u\in \tilde {\mathcal G}'_N\setminus \{w\})$. First, by a union bound over all $u\in V_N$ and using the bounds in \eqref{f(u)-bound} and \eqref{variance-bound}, we have (first take $C$, then $C_1$ to be sufficiently large)
	\begin{equation}\label{all-positivity-1}
		\P_M(\tilde {\eta}'_{N+2, u}\geq f(u)/2  \text{ for all }  u\in V_N)\geq 1/2\,.
	\end{equation}
	Conditioned on the values $\tilde {\eta}'_{N+2, u}$ for all $u\in V_N$, for each segment $e'=(x,y)\in E'_{N,w}$, we have (here $d(x,y)$ denotes the distance between $x$ and $y$ in the metric graph $\tilde {\mathcal G}'_N$)
	$$\P(\tilde {\eta}'_{N+2, u}=0 \text{ for some } u\in e'\mid M, \mathcal F_{V_N})= \mathrm{e}^{-\tilde {\eta}'_{N+2, x}\tilde {\eta}'_{N+2, y}\cdot \frac 1 {d(x,y)}}\leq \mathrm{e}^{-\frac 1 {16} f(x)f(y)}$$
	on the event $\{\tilde {\eta}'_{N+2, u}\geq f(u)/2  \text{ for all }  u\in V_N\}$. By another union bound over all segments $e'\in E'_{N,w}$ and using \eqref{f(u)-bound}, we have on the same event (recall that $C$ is large),
	\begin{equation}\label{all-positivity-2}
		\P(\tilde {\eta}'_{N+2, u}>0 \text{ for all } u\in e' \text { and all } e'\in E'_{N,w}\mid  M, \mathcal F_{V_N})\geq 1/2\,.
	\end{equation}
	Combining \eqref{all-positivity-1} and \eqref{all-positivity-2} we have $\P_M(\tilde {\eta}'_{N+2, u}>0 \text{ for all } u\in \tilde {\mathcal G}'_N\setminus \{w\})\geq 1/4$, and therefore
	\begin{equation}\label{all-positivity-3}
		\P_M(\tilde {\eta}'_{N+2, u}>0 \text{ for all } u\in \Pi_N\setminus \{w\})\geq 1/4\,.
	\end{equation}
	Combining \eqref{FKG}, \eqref{mean-on-positivity} and \eqref{all-positivity-3},  we can complete the proof of the lemma by choosing $C_0=4C_3$.\qed
\end{proof}

\begin{lemma}\label{X3-conditional-variance}
	There exists a constant $C_4>0$ such that
	\begin{align*}
		\var (X_3\mid H_=, H_+, H_-)
		&\leq C_4 /\log N\,.
	\end{align*}
\end{lemma}

\begin{proof}
	We first claim that
	\begin{equation}\label{Brascamp-Lieb}
		\var (X_3\mid H_=, H_+, H_-)\leq \var (X_3\mid H_=)=\var (X_3\mid \mathcal{F}_{J_1\cup J_2})\,.
	\end{equation}
	To show this, we use the Brascamp-Lieb inequality \cite{BL76} (see also \cite[Appendix B.2]{giacomin2001aspects}). Denote by $\mu$ the law of $Z$ where $Z$ is distributed as $\{\tilde {\eta}'_{N, u}:u\in \Pi_N\setminus(J_1\cup J_2)\}$ conditioned on $H_=$. Then $Z$ is a finite dimensional Gaussian vector. Let $m$ and $A$ be its mean vector and covariance matrix, respectively. The density of $\mu$ is of the form $\mu(\,dr)\propto\exp(-\frac 1 2 (r-m)\cdot A^{-1}(r-m))\,dr$. For any $q>0$, consider the measure
	$$\mu^{(q)}(\,dr)\propto\exp (-\sum_{u\in  I^{>\lambda}_k}U^{(q)}(r_u-\lambda-\frac 1 N)-\sum_{u\in I^{<-\lambda}_k}V^{(q)}(r_u+\lambda+\frac 1 N))\mu(\,dr)\,,$$
	where $U^{(q)}$ and $V^{(q)}$ are as defined in \eqref{def-UVW}. Since the second order derivatives of $U^{(q)}$ and $V^{(q)}$ are both nonnegative, we see that the density of $\mu^{(q)}$ is of the form $\mu^{(q)}(\,dr)=\exp(-H(r))\,dr$ where $\inf_r\text{Hess}(H)(r)\geq \frac 1 2 A^{-1}$. Therefore, by the Brascamp-Lieb inequality, for the random vector $Y^{(q)}\sim\mu^{(q)}$ and for every $l\in \mathbb R^{|\Pi_N\setminus(J_1\cup J_2)|}$, we have
	$\var(l\cdot Y^{(q)})\leq \var (l\cdot Z)$. 
	Since as $q\to\infty$, the law of $Y^{(q)}$ (i.e. $\mu^{(q)}$) converges weakly to the law of $Z$ conditioned on $H_+$ and $H_-$, we see that
	$$\var(l\cdot Z\mid H_+, H_-)\leq \var (l\cdot Z)\,.$$
	Note that
	$$\var (X_3\mid H_=, H_+, H_-)=\var (\frac {1}{ |\partial V_{(1+\beta) N/2}|}\sum_{v\in  \partial V_{(1+\beta) N/2}}\sum_{u\in J_3} \widetilde{\mathrm{Hm}}'(v, u; \temp \cup \partial V_N) \cdot \tilde {\eta}'_{N, u}\mid H_=, H_+, H_-)\,.$$
	Thus, by setting $l_u=\frac {1}{ |\partial V_{(1+\beta) N/2}|}\sum\limits_{v\in  \partial V_{(1+\beta) N/2}} \widetilde{\mathrm{Hm}}'(v, u; \temp \cup \partial V_N)$ for $u\in J_3$ and 0 otherwise, this gives the inequality \eqref{Brascamp-Lieb}.
	
	Now let us define
	$$U_1=\{u_1\in J_3:|u_1-u|\geq (\log N)^{10}\text{ for all } u\in J_4\}$$
	and for $u_1\in U_1$, define
	$$U_2(u_1)=\{u_2\in J_3:|u_1-u_2|\geq (\log N)^{10}\}\,.$$
	For $u_1,u_2\in J_3$, we say a pair $(u_1,u_2)$ is good if $u_1\in U_1$ and $u_2\in U_2(u_1)$. We can expand the right hand side of \eqref{Brascamp-Lieb} as follows (where we write $\mathcal I_{k, N} = \temp \cup \partial V_N$):
	\begin{eqnarray}\label{expand-var-X3}
		&&\var (\sum_{v\in  \partial V_{(1+\beta) N/2}}\sum_{u\in J_3} \widetilde{\mathrm{Hm}}'(v, u; \mathcal I_{k, N}) \cdot \tilde {\eta}'_{N, u}\mid \mathcal{F}_{J_1\cup J_2})\nonumber\\
		&=&\sum_{v_1,v_2\in  \partial V_{(1+\beta) N/2}}\sum_{u_1\in J_3\setminus U_1}\sum_{u_2\in J_3}\widetilde{\mathrm{Hm}}'(v_1, u_1; \mathcal I_{k, N})\widetilde{\mathrm{Hm}}'(v_2, u_2;\mathcal I_{k, N})G_{\tilde {\mathcal G}'_N\setminus  (J_1\cup J_2)}(u_1,u_2)\nonumber\\
		&&+\sum_{v_1,v_2\in  \partial V_{(1+\beta) N/2}}\sum_{u_1\in U_1}\sum_{u_2\in J_3\setminus U_2(u_1)}\widetilde{\mathrm{Hm}}'(v_1, u_1; \mathcal I_{k, N})\widetilde{\mathrm{Hm}}'(v_2, u_2; \mathcal I_{k, N})G_{\tilde {\mathcal G}'_N\setminus  (J_1\cup J_2)}(u_1,u_2)\nonumber\\
		&&+\sum_{v_1,v_2\in  \partial V_{(1+\beta) N/2}}\sum_{(u_1,u_2)\text { is good }}\widetilde{\mathrm{Hm}}'(v_1, u_1; \mathcal I_{k, N} )\widetilde{\mathrm{Hm}}'(v_2, u_2;\mathcal I_{k, N})G_{\tilde {\mathcal G}'_N\setminus  (J_1\cup J_2)}(u_1,u_2)\,.
	\end{eqnarray}
	
	Recall that we have $|J_4|\leq N \mathrm{e}^{-(\log N)^{\chi}}$. By a simple volume consideration, we have $|J_3\setminus U_1|\leq N \mathrm{e}^{-(\log N)^{\chi}}(\log N)^{21}$ and $|J_3\setminus U_2(u_1)|\leq N \mathrm{e}^{-(\log N)^{\chi}}(\log N)^{21}$ for $u_1\in U_1$. Therefore, for any $v_1,v_2\in  \partial V_{(1+\beta) N/2}$, we can deduce from Lemma~\ref{lem-Makarov} that (see \eqref{eq-harmonic-measure-J-4} for a similar derivation)
	\begin{equation}\label{U_1}
		\widetilde{\mathrm{Hm}}'(v_1, J_3\setminus U_1; \temp \cup \partial V_N)=o(\log N)^{-10}
	\end{equation}
	and
	\begin{equation}\label{U_2}
		\widetilde{\mathrm{Hm}}'(v_2, J_3\setminus U_2(u_1); \temp \cup \partial V_N)=o(\log N)^{-10}\,.
	\end{equation}
	It is well known that for a constant $C_5>0$, we have for any $u_1,u_2\in \tilde {\mathcal G}'_N$
	\begin{equation}\label{cov-log N}
		G_{\tilde {\mathcal G}'_N\setminus (J_1\cup J_2)}(u_1,u_2)\leq G_{\tilde {\mathcal G}'_N}(u_1,u_2) \leq C_5\log N\,.
	\end{equation}
	We claim that there exists a constant $C_6>0$, such that if $(u_1,u_2)$ is good, then 
	\begin{equation}\label{good}
		G_{\tilde {\mathcal G}'_N\setminus (J_1\cup J_2)}(u_1,u_2)\leq\frac {C_6} {\log N}\,.
	\end{equation}
Provided with \eqref{good}, we can substitute \eqref{U_1}, \eqref{U_2}, \eqref{cov-log N} and \eqref{good} into \eqref{expand-var-X3} and complete the proof of the lemma.

Therefore, it remains to  prove \eqref{good}. The key ingredient in proving \eqref{good} is that if $B^{\tilde {\mathcal G}'_N}$ is started at $u_1$, then the probability that it goes $(\log N)^{10}$ away from $u_1$ before hitting $J_1\cup J_2$ is, say, less than $\frac {C_7} {(\log N)^2}$ for a constant $C_7>0$.	To show this, we use the Beurling's estimate (see, e.g., \cite[Theorem 6.8.1]{LL10}). We observe that $J_1\cup J_2\cup J_4$ (as the ``outer boundary'' of $\temp$) is a $*$-connected set (where we regard two vertices as neighbors if their $\ell_\infty$-distance is at most 1) with diameter of order $N$, and thus
	$$V\deq\{v\in V_N:|v-u|\leq 1 \text{ for some } u\in J_1\cup J_2\cup J_4\}$$
	is a connected set with diameter of order $N$. In particular, by the  definition of $J_3$ we have $J_3\subseteq V$. By Beurling's estimate, once $B^{\tilde {\mathcal G}'_N}$ is at $v\in V$, it will hit $V$ again before going $(\log N)^6$ away from $v$, with probability at least $1-\frac {C_8} {(\log N)^3}$ (where $C_8>0$ is an absolute constant). Thus, if  $B^{\tilde {\mathcal G}'_N}$ is started at $u_1\in J_3$, then with probability at least $1-\frac {C_8} {(\log N)^2}$, it will hit $V$ at least $\log N$ times, before going $(\log N)^7$ away from $u_1$. However, it is clear that if $B^{\tilde {\mathcal G}'_N}$ is at $v\in V$, then it has at least constant probability ($\geq 1/32$) to hit $J_1\cup J_2\cup J_4$ before (or at) hitting a neighbor of $v$. Therefore, at these $\log N$ times that $B^{\tilde {\mathcal G}'_N}$ hits $V$ (before going $(\log N)^7$ away from $u_1$), it has at least $1-\frac 1 {N^{\log {\frac {32}{31}}}}$ probability to hit $J_1\cup J_2\cup J_4$ at least once in the following step, and since $J_4$ is $(\log N)^{10}$ away from $u_1$, it must hit $J_1\cup J_2$. That is to say, the probability that $B^{\tilde {\mathcal G}'_N}$ hits $J_1\cup J_2$ before going $(\log N)^{10}$ away from $u_1$ is at least $(1-\frac {C_8} {(\log N)^2})(1-\frac 1 {N^{\log {\frac {32}{31}}}})$, which is greater than $1-\frac {C_7} {(\log N)^2}$ for any $C_7>C_8$.  Now since $|u_1-u_2|\geq (\log N)^{10}$, and the expected number of visits of $u_2$ by $B^{\tilde {\mathcal G}'_N}$ is by \eqref{cov-log N} at most $C_5\log N$, we see that \eqref{good} is valid with $C_6=C_7C_5$. \qed
\end{proof}

\bigskip

\begin{acknowledgement}
We are most grateful to Greg Lawler for numerous stimulating discussions, including introducing the percolation problem on random walk loop soups and explaining his work on Makarov's theorem. We thank Marek Biskup for helpful discussions, intended for a different project but turned out relevant for the current article. We also thank Hubert Lacoin for helping locating \cite{Dunlop1992}.  As always, we thank Steve Lalley for his constant encouragement and support, and many useful discussions. The authors warmly acknowledge support by NSF grant DMS-1455049, DMS-1757479 and an Alfred Sloan fellowship. 
\end{acknowledgement}

\end{document}